\documentclass{article}

\usepackage{authblk}
\usepackage[utf8]{inputenc}

\usepackage{amsmath,amsfonts,amssymb,amsthm}

\usepackage{color}
\usepackage{vmargin}
\setmarginsrb{2cm}{1cm}{2cm}{3cm}{1cm}{1cm}{2cm}{2cm}
\numberwithin{equation}{section}

\renewcommand{\div}{\operatorname{div}}

\newcommand{\supp}{\operatorname{supp}}

\renewcommand{\leq}{\leqslant}

\renewcommand{\geq}{\geqslant}

\newcommand{\vstar}{v_*}
\newcommand{\vperp}{v_*^{\mathcal C}}

\usepackage{amssymb,amsmath,amsthm,scalefnt,amscd}
\usepackage{color,graphicx}

\newtheorem{theorem}{Theorem}[section]
\newtheorem{corollary}[theorem]{Corollary}

\newtheorem{lemma}[theorem]{Lemma}
\newtheorem{proposition}[theorem]{Proposition}
\newtheorem{remark}[theorem]{Remark}

\title{Optimal boundary control for steady motions of a self-propelled body in a Navier-Stokes liquid} 
\author[1]{Toshiaki Hishida\thanks{T. Hishida is partially supported by Grant-in-Aid for Scientific Research, 18K03363, from JSPS.}}
\affil[1]{Graduate School of Mathematics, Nagoya University, Nagoya 464-8602 Japan
\tt{hishida@math.nagoya-u.ac.jp}}
\author[2]{Ana Leonor Silvestre\thanks{A. L. Silvestre acknowledges the financial support of the Portuguese FCT - Funda\c{c}\~ao para a Ci\^{e}ncia e a Tecnologia, through the projects UIDB/04621/2020 and UIDP/04621/2020 of CEMAT/IST-ID.}}
\affil[2]{CEMAT and Department of Mathematics,
Instituto Superior T\'{e}cnico, Universidade de Lisboa,
Av. Rovisco Pais 1, 1049-001 Lisboa, Portugal
\tt{Ana.Silvestre@math.tecnico.ulisboa.pt}}
\author[3]{Tak\'eo Takahashi\thanks{T. Takahashi is partially supported by the project IFSMACS ANR-15-CE40-0010, financed by the French Agence Nationale de la Recherche.}}
\affil[3]{Universit\'e de Lorraine, CNRS, Inria, IECL, F-54000 Nancy, France
\tt{takeo.takahashi@inria.fr}}

\date{\today}

\begin{document}
\maketitle

\abstract{
Consider a rigid body ${\mathcal S} \subset {\mathbb R}^3$ immersed in an infinitely extended Navier-Stokes liquid and the motion of the 
 body-fluid interaction system described from a reference frame attached to ${\mathcal S}$. We are interested in steady motions of this coupled system, where the region occupied by the fluid is the exterior domain $\Omega = {\mathbb R}^3 \setminus {\mathcal S}$. This paper deals with the problem of using boundary controls $v_*$, acting on the whole $\partial\Omega$ or just on a portion $\Gamma$ of $\partial\Omega$, to generate a self-propelled motion of ${\mathcal S}$ with a target velocity $V(x):=\xi+\omega \times x$ and  to  minimize the drag about ${\mathcal S}$. Firstly, an appropriate drag functional is derived from the energy equation of the fluid and the problem is formulated as an optimal boundary control problem.  Then the minimization problem is solved for localized controls, such that supp $v_*\subset \Gamma$, and for tangential controls, i.e, $v_*\cdot n|_{\partial \Omega}=0$, where $n$ is the outward unit normal to $\partial \Omega$. We prove the existence of optimal solutions, justify the G\^ateaux derivative of the  control-to-state map, establish the well-posedness of the corresponding adjoint equations and,
 finally, derive the first order optimality conditions. The results are obtained under smallness restrictions on the objectives $|\xi|$ and $|\omega|$ and on the boundary controls.
}

\vspace{1cm}


%
%
%
%
%
%

\tableofcontents

\section{Introduction}

Consider a rigid body $\mathcal{S}$ moving by self-propulsion through an infinitely extended incompressible viscous fluid $\mathcal{F}$. This means that  the total net force and torque, external to the system $\{ \mathcal{S}, \mathcal{F} \}$, acting on $\mathcal{S}$, are identically zero. Since the shape of the body is constant, the self-propelled motion of $\mathcal{S}$ through $\mathcal{F}$ is due to the boundary values for the velocity 
of the system at the boundary of $\mathcal{S}$, which are prescribed relative to $\mathcal{S}$. For instance, the propulsion may be produced by drawing fluid inwards across portions of the boundary and by expelling it from others, or by the tangential motion of certain portions of the boundary, as by belts. The propulsion of jet planes and submarines or of minute organisms like ciliates and flagellates can be considered of this nature.

In this context, the set of equations describing the motion of $\{ \mathcal{S}, \mathcal{F} \}$, in a reference frame attached to $\mathcal{S}$, is 
\begin{gather}
-\div \sigma(v,p)+  (v-V)\cdot \nabla v+ \omega\times v =0 \quad \text{in} \ \Omega \label{000}\\
\div v =0 \quad \text{in} \ \Omega \label{001}\\
v=V+v_* \quad \text{on}\ \partial \Omega\label{002}\\
\lim_{|x|\to \infty} v(x) = 0\label{003}\\
m \xi\times \omega + \int_{\partial \Omega} \left[-\sigma(v,p)n +  \left(v_*\cdot n \right) (v_* + V + \omega \times x)
  \right] \ d \gamma =0
\label{903}\\
(I \omega)\times \omega + \int_{\partial \Omega} x\times \left[-\sigma(v,p)n +  \left(v_*\cdot n \right) (v_* + V + \omega \times x) \right] \ d \gamma =0
\label{004}
\end{gather}
where the quantities $v=v(x)$ and $p=p(x)$ represent, respectively, the velocity field and the pressure of the liquid and 
\begin{equation}\label{tt01}
V(x)=\xi + \omega \times x, \quad x\in \mathbb{R}^3,
\end{equation}
represents the velocity of the solid, as seen by an observer attached to $\mathcal{S}$. We assume that the density of the fluid is constant and equal to 1. Moreover, we denote by $\sigma(v,p)$ the Cauchy stress tensor defined by
$$
\sigma(v,p):=2 D(v) - p {\mathbb I}_3,
$$
where the viscosity of the fluid is also assumed equal to 1,  ${\mathbb I}_3$ is the $3\times 3$ identity matrix and $D(v)$ is the symmetric part of the gradient, 
\[
D(v):=\frac{1}{2}\left(\nabla v+(\nabla v)^\top\right).
\]
Due to the incompressibility condition \eqref{001}, we can use  the relation $\div \sigma(v,p)= \Delta v-\nabla p$ in equation \eqref{000}. Moreover, in equation \eqref{000} and in what follows, for sufficiently regular vector fields $u$ and $v$, $u\cdot \nabla v$ is the vector field with components $(u\cdot \nabla v)_i = \sum_{j=1}^3 u_j \frac{\partial v_i}{\partial x_j}$. The set $\mathcal{S}$ representing the rigid body is a compact simply connected set, with non empty interior, so that $\Omega=\mathbb R^3\setminus {\cal S}$ is a three-dimensional exterior domain. In \eqref{903}--\eqref{004} and in what follows, $n$ denotes the outward unit normal to $\partial\Omega$, which is well defined provided $\partial \Omega$ is locally Lipschtizian.  Throughout this paper, nevertheless, $\partial\Omega$ is assumed to be of class $C^3$ (which is in fact needed in Theorem \ref{Tmain}.)
We will assume that the center of gravity of the rigid body is located at the origin, that is $\int_{\mathcal{S}} x \ dx=0$, and define the mass and inertia matrix by
$$
m:=\int_{\mathcal{S}} \ dx, \quad  
I:=\int_{\mathcal{S}} (|x|^2 {\mathbb I}_3 - x\otimes x) \ dx,
$$
where we are assuming that the density of the rigid body is constant and equal to 1. The boundary values $v_*$ represent the thrust velocity, responsible for the motion of $\mathcal S$. The model \eqref{000}--\eqref{004} is inspired by Galdi
\cite{Galdi1997,Galdi1999,GRev}, the equation \eqref{903} having been obtained as the net force exerted by the fluid on the solid and the equation \eqref{004} being the corresponding balance of torques.

While moving through the fluid, ${\mathcal S}$ will experience a drag, i.e., a net force in the direction of flow, due to the pressure and shear stress on  the surface of ${\mathcal S}$, which tends to slow down its motion. In this paper, our aim is to use the Dirichlet boundary data $v_*$ not only to self-propell ${\mathcal S}$, but also to minimize the work needed to overcome the drag exerted by the fluid on $\mathcal{S}$
\begin{equation}
W (v,p):= \int_{\partial\Omega} v \cdot \big(\sigma(v,p) n\big) \ d\gamma = \int_{\partial\Omega} (v_* + V)  \cdot \big(\sigma(v,p) n\big)\  d\gamma
\label{workdrag}
\end{equation}
when it performs a motion with the objective velocity $V$.

In our previous work \cite{HST}, by the present authors, 
we have solved the control problem \eqref{000}--\eqref{004} for $v_*$ in 
finite-dimensional control spaces: $\mathcal{C}_\chi$ for localized controls and $\mathcal{C}_\tau$ for tangential controls (see the definition of those spaces in \eqref{controlspace} and \eqref{controlspacetau}).

In this paper, our aim is to consider controls $v_*$ in infinite-dimensional spaces and to show the existence of an optimal control which minimizes  
the drag functional given by \eqref{workdrag}. Similar problems have been solved in \cite{FGH0,FGH} in the unsteady case without spin of the body by taking for the state equations only the classical Navier-Stokes system. A control problem for the swimming of microscopic organisms was solved in \cite{MTT}. Here, we also want to characterize the minima by means of the adjoint system. The optimality systems can serve as the basis for computing approximations to optimal solutions numerically as in \cite{D}. The 
 previous results of \cite{HST} will be used to construct a corrector of the control $v_*$ which ensures that the self-propelling conditions are satisfied (see Theorem \ref{Tmain}). 

The main difficulties we have to overcome are the usual ones when dealing with exterior domains \cite{G,Galdi1997}, namely the need of knowing the asymptotic behavior of solutions as well as the presence of terms associated with the rotation of the solid. For example, a good knowledge of the rate of decay of the velocity field is crucial to establish the energy equation for a flow \eqref{000}--\eqref{003}, which, in turn, permits to write the cost functional \eqref{workdrag} in a more convenient form, where, in particular, the pressure is not present. 
The definition of the Lagrangian is another delicate issue in the analysis of the control problems, which benefits from the extra regularity at infinity provided by the self-propelling condition \eqref{903}, as justified in detail in \cite{HST}. In particular, we use the fact that $v\in L^2(\Omega)$, which in general does not hold in exterior domains. 
 We develop even $L^2$-estimate of the velocity $v$, which seems to be of independent interest, and apply it to justify the G\^ateaux differentiability of the control-to-state map with values in $W^{2,2}(\Omega)$ (especially in $L^2(\Omega)$), which plays a role to deduce the optimal condition at the final stage.
For technical reasons, as in \cite{FGH}, we need to impose some restrictions on the size of the control $v_*$ and on the size of the objective velocity components $\xi$ and $\omega$.  

The plan of the paper is the following. In Section \ref{sec_not}, we gather some notation and 
 an auxiliary result. The new $L^2$-estimates for the linearized problem associated with \eqref{000}--\eqref{003}, which are consequence of a condition on the net force exerted by the fluid to the rigid body similar to \eqref{903}, are deduced in Section \ref{newL2estimates}. In Section \ref{sec_pb}, we precisely state  our optimal control problems associated with system \eqref{000}--\eqref{004}. Existence of minima to those control problems is shown in Section  \ref{sec_exi} . 
Then in order to characterize minima, we study, in Section \ref{sec_reg}, the regularity of the 
 control-to-state mapping, more specifically, its G\^ateaux differentiability. Finally, in Section \ref{sec_main}, we obtain optimality conditions of order 1 for our problems. 
These conditions require the well-posedness of the adjoint system.

\section{Notation and auxiliary results}\label{sec_not}
Throughout the paper we shall use the same font style to denote scalar, vector and tensor-valued functions and corresponding function spaces. We use the usual notation to denote Lebesgue and Sobolev spaces on a domain ${\mathcal A}$, namely,
$L^q({\cal A})$ and $W^{m,q}({\cal A})$,
with norms $\|. \|_{q,{\mathcal A}}$  and $\|. \|_{m,q,{\mathcal A}},$ respectively. By $W^{m-\frac{1}{q},q}(\partial{\mathcal A})$ we indicate the trace space on the (sufficiently)
smooth boundary $\partial{\mathcal A}$ of ${\cal A}$, for functions from  $W^{m,q}({\mathcal A}),$ equipped with the usual norm $\| . \|_{m-\frac{1}{q},q,\partial{\mathcal A}}.$ The homogeneous Sobolev space of order $(k,q)\in {\mathbb N} \times (1,\infty)$ is defined by
\[
D^{k,q}({\mathcal A})
:=\{u\in L^1_{loc}({\cal A}) \ ;\ D^\alpha u\in L^q({\cal A})\;
\mbox{for any multi-index $\alpha$ with $|\alpha|=k$}\}
\]
with associated seminorm $|u|_{k,q,{\mathcal A}}=\sum_{|\alpha|=k}\|D^\alpha u\|_{q,{\mathcal A}}$. 
The space $D^{1,q}_0({\mathcal A})$ is the closure of $C^\infty_0({\mathcal A})$ in $D^{1,q}({\mathcal A})$.  In particular, the dual space of $D^{1,2}_0({\mathcal A})$ defined for a domain
${\mathcal A}$, $D^{-1,2}({\mathcal A})$ with norm $|\cdot |_{-1,2,{\mathcal A}}$, will be used in this work.

In what follows, we can assume that $\Omega$ is the exterior domain $\mathbb R^3\setminus {\mathcal S}$. By ${\mathcal D}(\Omega)$ we denote the space of $C_0^\infty(\Omega)^3$-functions which are divergence free.  For a vector or second-order tensor field $G$, $\alpha\geq 0$ and a positive function $\varpi$ defined on $\Omega$, we use the notation
\begin{equation}
\lceil G \rceil_{\alpha,\varpi,\Omega}
:=\sup_{x \in \Omega}[\varpi(x)^{\alpha}|G(x)|],
\end{equation}
and denote by $L^\infty_{\alpha,\varpi}(\Omega)^3$ and $L^\infty_{\alpha,\varpi}(\Omega)^{3\times 3}$ the spaces of those $G \in L^\infty(\Omega)^3$ or $G \in L^\infty(\Omega)^{3\times 3}$ such that the norm $\lceil G \rceil_{\alpha,\varpi,\Omega}$ is finite. To alleviate the notation, we will write only $L^\infty_{\alpha,\varpi}(\Omega)$ instead of $L^\infty_{\alpha,\varpi}(\Omega)^3$ and $L^\infty_{\alpha,\varpi}(\Omega)^{3\times 3}$. 

For $R>0$, we denote by $B_R$ the open ball $
B_R:=\{x \in {\mathbb R}^3 ; |x| <R \}$ and by $A_{R_1,R_2}$ ($R_2>R_1$) the spherical-annulus domain $A_{R_1,R_2}:=\{x \in {\mathbb R}^3 ; R_1< |x| <R_2 \}$.
If ${\mathbb R}^3 \setminus \Omega \subset B_\varrho$ for some $\varrho >0$,
we set $\Omega_R  :=  \Omega \cap B_R $ and $\Omega^R := \Omega \setminus \overline{B_R}$ for $R >\varrho$.

Now we collect a number of useful results concerning the generalized Oseen system  \begin{gather}
-\div \sigma(u,q) - V \cdot \nabla u + \omega\times u  = f \quad \text{in} \ \Omega \label{lis1.3}\\
\div u =0 \quad \text{in} \ \Omega \\
u = u_*  \quad \text{on}\ \partial \Omega \\
\lim_{|x|\to \infty} u(x) = 0 \label{lis1.4}
\end{gather}   
where $V$ is given by \eqref{tt01}.
For the case of external force $f=\mbox{div $F$}$ with $F$ satisfying some
anisotropic pointwise estimate, see \cite[Proposition 2.1]{HST}. The following proposition provides a solution with several properties when less assumptions are imposed on $f$.
\begin{proposition}\label{P23}
If $\Omega$ is of class $C^2$, $f \in L^2(\Omega)\cap D^{-1,2}(\Omega)$ and $u_* \in W^{3/2,2}(\partial \Omega)$, 
then there is a unique 
solution $(u,q)$
to \eqref{lis1.3}--\eqref{lis1.4} such that 
$$
|u|_{1,2,\Omega}  + |u|_{2,2,\Omega}+\|u\|_{\infty,\Omega} + \|q\|_{1,2,\Omega}
\leq C(B,\Omega) \big(\|f\|_{2,\Omega} + |f|_{-1,2,\Omega} + \| u_*\|_{3/2,2,\partial\Omega}\big),
$$
where $B>0$ is such that $|\xi|,|\omega|\leq B$.  
\end{proposition}
\begin{proof}For the sake of the readers, we give 
the main steps of the proof although the results are more or less known. First, existence of a weak solution $u\in D^{1,2}(\Omega)$ of \eqref{lis1.3}--\eqref{lis1.4} is obtained by applying \cite[Theorem VIII.1.2, p. 501]{G}. Note that for this step, it is sufficient that $\Omega$ is a locally Lipschitz domain, 
$f \in D^{-1,2}(\Omega)$
and $u_* \in W^{1/2,2}(\partial \Omega)$. The pressure $q\in L^2_{loc}(\overline{\Omega})$ is recovered by \cite[Lemma VIII.1.1, p. 500]{G}. This weak solution satisfies the following estimates:
\begin{equation}
\| u \|_{2,\Omega_R} + \| \nabla u \|_{2,\Omega} 
+\|u\|_{6,\Omega}
\leq C(R,\Omega) \big[|f|_{-1,2,\Omega} + (1+|\xi| + |\omega|)\| u_*\|_{1/2,2,\partial\Omega}\big],
\label{ubound}
\end{equation}
\begin{equation}
 \|q-\overline{q}_R\|_{2,\Omega_R}
\leq C(R,\Omega) \big[|f|_{-1,2,\Omega} + (1+|\xi| + |\omega|)\| \nabla u \|_{2,\Omega}\big],
\label{qbound}
\end{equation}
for $R$ sufficiently large, where $\overline{q}_R:=\frac{1}{|\Omega_R|}\int_{\Omega_R}q$.
Uniqueness of the weak solution (up to constants for the pressure, however, we will actually single out it later) is given by \cite[Lemma VIII.2.3, p. 512]{G}  (see also Theorem VIII.2.1).

Then, we apply \cite[Lemma VIII.2.1, p. 505]{G} to obtain 
\begin{equation}
\|D^2u\|_{2,\Omega^r}\leq C(r,\rho,B)\big(\|f\|_{2,\Omega^\rho}+\|\nabla u\|_{2,\Omega}\big)
\label{u-second}
\end{equation}
with $r>\rho$ provided $f \in L^{2}(\Omega^\rho)$ for some $\rho>0$.
Moreover, \cite[Lemma VIII.6.1, p. 554]{G} yields estimates for the second derivatives of $u$ and first derivatives of $q$ in a bounded domain $\Omega_r$, provided $f \in L^{2}(\Omega_R)$ for some $R>r$,
$u_* \in W^{3/2,2}(\partial \Omega)$
and $\Omega$ is of class $C^2$:
\begin{equation}
\| u \|_{2,2,\Omega_r} + 
\|\nabla q\|_{2,\Omega_r}
\leq C(r,R,\Omega,B) \big(\|f\|_{2,\Omega_R} + \| u_*\|_{3/2,2,\partial\Omega} + \| u \|_{2,\Omega_R} 
+\| q -\overline{q}_R \|_{2,\Omega_R}\big).
\label{nuq2Omega}
\end{equation}

From the equation \eqref{lis1.3} together with \eqref{ubound} and \eqref{u-second} it follows that
$$
\frac{|\nabla q(x)|}{1 + |x|} \leq  |\Delta u(x)| + |f(x)| + C(\xi,\omega)\left(|\nabla u(x)| + \frac{|u(x)|}{1+|x|}\right) \text{ a.e. } \Omega,
$$
and since the right-hand side of this inequality is square-summable near infinity (Hardy inequality is used for the last term), we have
\begin{equation}
\frac{\nabla q}{1 + |x|} \in  L^2(\Omega^r)
\label{growp}
\end{equation}
for every sufficiently large $r>0$.

Given $r>0$ large enough such that $\mathbb R^3\setminus\Omega\subset B_{r/2}$, we fix a cut-off function $\psi\in C_0^\infty(B_r; [0,1])$ satisfying $\psi(x)=1$ for $x\in B_{r/2}$. Given a pressure $q$, let us consider the pair $(\widetilde u, \widetilde q)=((1-\psi)u,(1-\psi)(q-\overline{q}_r))$, that obeys 
\[
-\div\sigma(\widetilde u,\widetilde q)-V\cdot\nabla\widetilde u+\omega\times\widetilde u=\widetilde f, \quad
\div\widetilde u=g \quad\mbox{in $\mathbb R^3$}
\]
with
\[
\widetilde f:=(1-\psi)f+2\nabla\psi\cdot\nabla u+(\Delta\psi+V\cdot\nabla\psi)u-(\nabla\psi)(q-\overline{q}_r), \quad
g:=-u\cdot\nabla\psi.
\]
Then
\begin{equation}
\Delta\widetilde q=\div (\widetilde f+\nabla g + gV) \quad\mbox{in $\mathbb R^3$}
\label{q-whole}
\end{equation}
and from \eqref{growp} we deduce $\nabla\widetilde q\in {\mathcal S}^\prime(\mathbb R^3)$ and thereby, $\widetilde q\in {\mathcal S}^\prime(\mathbb R^3)$. Every solution to \eqref{q-whole} within ${\mathcal S}^\prime(\mathbb R^3)$ is represented as
\begin{equation}
\widetilde{q} = Q + P, \qquad Q = - {\mathcal F}^{-1} \left(   \frac{i \zeta \cdot \hat h }{|\zeta|^2}   \right)
\label{qQ}
\end{equation}
where $P$ is a harmonic polynomial, ${\cal F}^{-1}$ denotes the Fourier inverse transform and $$h:= \widetilde f+\nabla g + g V
\in L^2(\mathbb R^3).$$ Since $\nabla Q \in L^2(\mathbb R^3)$, which follows from $h \in L^2(\mathbb R^3)$, one needs $(\nabla P)/(1+|x|) \in  L^2(\Omega^r)$ to accomplish \eqref{growp} and this, in turn, is possible only if $P$ is a constant. We thus obtain 
\begin{equation}
 \nabla   \widetilde q =   \nabla Q
 \label{gradq}
 \end{equation}
together with the estimate 
\begin{equation}
\|\nabla q\|_{2,\Omega^r}\leq \|\nabla\widetilde q\|_{2,\mathbb R^3}
\leq \|\widetilde f+\nabla g+gV\|_{2,\mathbb R^3}
\leq C(r,B)\big(\|f\|_{2,\Omega}+\|u\|_{1,2,\Omega_r}+\|q-\overline{q}_r\|_{2,\Omega_r}\big).
\label{q-far1}
\end{equation}
To deduce this estimate, we have used  Plancherel Theorem. Combining \eqref{q-far1} with \eqref{nuq2Omega} leads to $\nabla q\in L^2(\Omega)$ and, therefore, there is a constant $a\in\mathbb R$ such that
$\|q-a\|_{6,\Omega}\leq C\|\nabla q\|_{2,\Omega}$. We now single out the pressure $q$ with $a=0$, that is, $q\in L^6(\Omega)$. Then we have
\begin{equation}
\|q\|_{2,\Omega_r}\leq C(r)\|q\|_{6,\Omega_r}\leq C(r)\|\nabla q\|_{2,\Omega}.
\label{q-loc}
\end{equation}
With this pressure at hand, let us go back to the cut-off procedure above and \eqref{q-whole}.

Considering the pairing
$\langle h, \varphi\rangle_{\mathbb R^3}$ with $\varphi\in C_0^\infty(\mathbb R^3)$ and noting $\|\varphi\|_{2,\Omega_r}\leq C\|\nabla\varphi\|_{2,\mathbb R^3}$ by the same reasoning as above, we find that $ h \in D^{-1,2}(\mathbb R^3)$ by duality. By \cite[Lemma 2.2]{KS1991} (see also \cite[Theorem II.8.2, p. 112]{G}), there is a tensor field $H \in L^2({\mathbb R}^3)$ such that 
$\langle h, \varphi\rangle_{\mathbb R^3}$ is given by 
$$\langle h, \varphi\rangle_{\mathbb R^3} = \int_{{\mathbb R}^3}H : \nabla \varphi dx = - \langle \div H , \varphi \rangle_{(C_0^\infty({\mathbb R}^3))' \times C_0^\infty({\mathbb R}^3)}, \, \forall \varphi \in C_0^\infty({\mathbb R}^3),
$$ 
and the following estimate holds:
\begin{equation}
\| H \|_{2, \mathbb R^3} = |h|_{-1,2,{\mathbb R}^3} =|\widetilde f+\nabla g+gV|_{-1,2,{\mathbb R}^3}
 \leq C(r,B)\big(|f|_{-1,2,\Omega}+\|u\|_{2,\Omega_r}+\|q-\overline{q}_r\|_{2,\Omega_r}\big).
\label{H-est}
\end{equation} 
From \eqref{qQ} and Plancherel Theorem, we get
\begin{equation}
Q = - {\mathcal F}^{-1} \left( \frac{(\zeta\otimes\zeta): \widehat H}{|\zeta|^2} \right), \qquad \| Q \|_{2, \mathbb R^3} \leq \| H \|_{2, \mathbb R^3}.
\label{QH}
\end{equation}
By \eqref{gradq} there is a constant $b\in\mathbb R$ such that 
$\widetilde q+b =Q \in L^2(\mathbb R^3)$ 
and, therefore, $q-\overline{q}_r+b \in L^2(\Omega^r)$, where $r>0$ is fixed 
 at the outset of the cut-off procedure.
However, it follows from $q\in L^6(\Omega)$ that $b=\overline{q}_r$.
As a consequence, we obtain $q\in L^2(\Omega^r)$ and, furthermore, by \eqref{H-est}--\eqref{QH}
\begin{equation}
\|q\|_{2,\Omega^r}\leq\|\widetilde q+\overline{q}_r\|_{2,\mathbb R^3}
=\|Q\|_{2,\mathbb R^3}
\leq C(r,B)\big(|f|_{-1,2,\Omega}+\|u\|_{2,\Omega_r}+\|q-\overline{q}_r\|_{2,\Omega_r}\big),
\label{q-far2}
\end{equation}
which together with \eqref{q-loc} implies that $q\in L^2(\Omega)$. We collect \eqref{ubound}--\eqref{nuq2Omega}, \eqref{q-far1}, \eqref{q-loc} and \eqref{q-far2} to conclude the desired estimates except the $L^\infty$-norm.

Finally, using \cite[Lemma 4.1]{Crispo}  and Sobolev inequalities,
we conclude that $u,\nabla u \in L^6(\Omega)$, $u\in L^\infty(\Omega)$  and 
$$
\| u \|_{\infty,\Omega} \leq C(\Omega,B) \left(\|f\|_{2,\Omega} + |f|_{-1,2,\Omega} + \| u_*\|_{3/2,2,\partial\Omega}\right).
$$
\end{proof}

\section{$L^2$-estimate of the solution to a linearized problem}\label{newL2estimates}

Given a rigid motion $V$ as in \eqref{tt01}, in this section, we consider better asymptotic behavior at infinity of the solution 
to the generalized Oseen system
\begin{equation}
- \div\sigma(u,q)
-V\cdot\nabla u+\omega\times u=\mbox{div $F$}, \quad 
\div u=0
\label{linear-ext}
\end{equation}
in an exterior domain $\Omega$ without explicitly specifying any boundary condition at  
$\partial\Omega$ when 
\begin{equation}
\begin{array}{rl}
N=0, &\qquad \mbox{if $\omega=0$},\\
\omega\cdot N=0, &\qquad \mbox{if $\omega\in\mathbb R^3\setminus\{0\}$}, 
\end{array}
\label{zero-force}
\end{equation}
where $N$ stands for the net force exerted by the fluid to the rigid body, that is,
\begin{equation}
N=N_{u,q}=\int_{\partial\Omega}
[\sigma(u,q)+u\otimes V-(\omega\times x)\otimes u+F]\,n\,d\gamma,
\label{net-force}
\end{equation}
which is well-defined as long as $(u,q)$ and $F$ are of class
\eqref{ext-force}--\eqref{class} below
(yielding $$[\sigma(u,q)+u\otimes V-(\omega\times x)\otimes u+F]\,n\in W^{-1/2,2}(\partial\Omega)$$as the normal trace).
 Note that \eqref{903} is equivalent to $N_{v,p}=0$ with $F=-v\otimes v$.
Under the condition \eqref{zero-force}, 
we know from asymptotic structure of the flow at infinity that 
$u\in L^2(\Omega)$, see \cite{FaHi} and \cite{Ky} especially for the case
$\omega\in\mathbb R^3\setminus\{0\}$,
however, to our knowledge useful estimates 
are not available so far in the literature.
For later use, 
we are aiming at deduction of the following $L^2$-estimate.
\begin{proposition}
\label{L2estimate}
Suppose
\begin{equation}
[x \mapsto (1+|x|)F(x)] \in L^2(\Omega).
\label{ext-force}
\end{equation}
Let $(u,q)$ be a solution to \eqref{linear-ext} of class
\begin{equation}
\nabla u, q\in L^2(\Omega), \qquad
u\in L^6(\Omega).
\label{class}
\end{equation}
Let us set
\[
\Phi=\int_{\partial\Omega}n\cdot u\,d\gamma.
\]

\begin{enumerate}
\item
Let $\omega=0$.
If $N=0$, then we have $u\in L^2(\Omega)$ subject to
\begin{equation}
\|u\|_{2,\Omega}
\leq C
\big[
(1+|\xi|)
\left(\|\nabla u\|_{2,\Omega}+|\Phi|\right)+\|q\|_{2,\Omega}+\|F\|_{2,\Omega}
\big]
+C'\||x|F\|_{2,\Omega}
\label{L2-est2}
\end{equation}
with some constants $C,\,C'>0$
which are independent of $u,q,F,\xi$ and $\Phi$.
\item
Let
$\omega\in\mathbb R^3\setminus\{0\}$.
If $\omega\cdot N=0$, then we have $u\in L^2(\Omega)$ subject to
\begin{equation}
\|u\|_{2,\Omega}\leq 
C K(u,q,F,\xi,\omega,\Phi)
+C'\||x|F\|_{2,\Omega}
\label{L2-est}
\end{equation}
with some constants $C,\,C'>0$
which are independent of $u,q,F,\xi,\omega$,
where
\begin{equation}
\begin{split}
&K(u,q,F,\xi,\omega,\Phi) =\left(
1+|\omega|^{-1/4}+\frac{|\omega\cdot\xi|^{1/2}}{|\omega|}
+\frac{|\omega\times\xi|}{|\omega|^2}
\right) \\
&\qquad \qquad \qquad \qquad \qquad
\cdot\big[
(1+|\xi|+|\omega|)
\big(\|\nabla u\|_{2,\Omega}+|\Phi|\big)+\|q\|_{2,\Omega}+\|F\|_{2,\Omega}
\big].
\end{split}
\label{quantity}
\end{equation}
\end{enumerate}
\label{square-est}
\end{proposition}

For the latter case $\omega\in\mathbb R^3\setminus\{0\}$,
it is also possible to deduce a bit different
estimate from \eqref{L2-est}:
\begin{equation}
\begin{split}
&\|u\|_{2,\Omega}  \leq C\left(
|\omega|^{-1/4}+\frac{|\omega\cdot\xi|^{1/2}}{|\omega|}
\right)
\big(|N|+|\omega||\Phi|\big)
+C'\||x|F\|_{2,\Omega}  \\
& \qquad \qquad \qquad +C\left(
1+\frac{|\omega\times\xi|}{|\omega|^2}
\right)
\big[
(1+|\xi|+|\omega|)
\big(\|\nabla u\|_{2,\Omega}+|\Phi|\big)+\|q\|_{2,\Omega}+\|F\|_{2,\Omega}
\big].
\end{split}
\label{L2-est3}
\end{equation}
As usual, by a cut-off procedure, the problem 
in exterior domains will be reduced to the one in the whole space. When 
$\omega\in\mathbb R^3\setminus\{0\}$,
we then use the Moggi-Chasles transform, see \cite{G} and \cite{GS1}, to modify the resulting problem in the following way
\begin{equation}
\begin{split}
&y=M\left(x-\frac{\omega\times\xi}{|\omega|^2}\right), \\
&{\mathfrak s}(y)= s \left(M^\top y+\frac{\omega\times\xi}{|\omega|^2}\right) \mbox{ for a scalar field},\\
&{\mathfrak v}(y)=M v\left(M^\top y+\frac{\omega\times\xi}{|\omega|^2}\right) \mbox{ for a vector field} ,\\
&{\mathfrak T}(y)=\Big(M T M^\top\Big)\left(M^\top y+\frac{\omega\times\xi}{|\omega|^2}\right) \mbox{ for a tensor field},
\end{split}
\label{changevariables}
\end{equation}
where $M\in \mathbb R^{3\times 3}$ being an orthogonal matrix
that fulfills $M\frac{\omega}{|\omega|}=e_1$, to obtain the generalized Oseen
system in which the direction of the translation is parallel to the
axis of rotation that becomes the $e_1$-direction.
Thus, let us consider the system
\begin{equation}
-\Delta v+\nabla p
-{\cal R}\partial_1v-|\omega|\left[(e_1\times y)\cdot\nabla v-e_1\times v\right]
=f, \quad 
\mbox{div }v=0 \quad \mbox{ in }\mathbb R^3_y
\label{linear-whole}
\end{equation}
within the class of tempered distributions,
where
\begin{equation}
{\cal R}=\frac{\omega\cdot\xi}{|\omega|}.
\label{parame}
\end{equation}
For the external force $f$ of a suitable class, we know that:
\begin{enumerate}
\item when $\omega=0$, to the classical Oseen system
\[
-\Delta v+\nabla p-\xi\cdot\nabla v=f, \quad 
\mbox{div $v$}=0  \mbox{ in } \mathbb R^3,
\]
we see that
\begin{equation}
\widehat v(\zeta)=\frac{1}{|\zeta|^2-i\xi\cdot\zeta}
\left(
\mathbb I_3-\frac{\zeta\otimes\zeta}{|\zeta|^2}
\right)
\widehat f(\zeta)
\label{formula2}
\end{equation}
 is a solution on the Fourier side.
\item when 
$\omega\in\mathbb R^3\setminus\{0\}$ 
\begin{equation}
\widehat v(\zeta)
=\int_0^\infty e^{-(|\zeta|^2-i{\cal R}\zeta_1)t}
O_\omega(t)^\top
\left(
\mathbb I_3-\frac{(O_\omega(t)\zeta)\otimes (O_\omega(t)\zeta)}{|\zeta|^2}
\right)
\widehat f(O_\omega(t)\zeta)\,dt
\label{formula}
\end{equation}
is a solution for \eqref{linear-whole} on the Fourier side, where
\[
O_\omega(t)=O(|\omega|t), \qquad
O(t)=
\left[
\begin{array}{ccc}
1 & 0 & 0 \\
0 & \cos t & -\sin t \\
0 & \sin t & \cos t
\end{array}
\right].
\]
\end{enumerate}
Relation \eqref{formula} is classical but we recall here the idea to obtain it: first we notice that the Fourier transform of $(v,p)$ satisfies
$$
\begin{array}{l}
(|\zeta|^2 - i{\cal R}\zeta_1 )\widehat v(\zeta) - |\omega|
\left[(e_1\times\zeta)\cdot\nabla_\zeta \widehat v(\zeta) -e_1\times \widehat v(\zeta)\right] - i \zeta  \widehat p(\zeta)=
\widehat f(\zeta), \medskip \\
 i \zeta \cdot \widehat v(\zeta) = 0, \quad \zeta \in {\mathbb R}^3.
\end{array}
$$
Eliminating the pressure, we find 
\[
(|\zeta|^2 - i{\cal R}\zeta_1 )\widehat v(\zeta) - |\omega|
\left[(e_1\times\zeta)\cdot\nabla_\zeta \widehat v(\zeta) -e_1\times \widehat v(\zeta)\right]=\left(
\mathbb I_3-\frac{\zeta \otimes \zeta}{|\zeta|^2}
\right)
\widehat f(\zeta), \quad \zeta \in {\mathbb R}^3.
\]
Then we define
\begin{equation}
\mathcal V(t,\zeta)= O_\omega(t) \widehat v(O_\omega(t)^\top\zeta) = O_\omega(t) \widehat v(O_\omega(-t)\zeta) 
\label{vtoV}
\end{equation}
and some standard computation yields that $\mathcal V(t,\zeta)$ is time-periodic and satisfies
$$
 \frac{\partial}{\partial t} \mathcal V(t,\zeta)  +  (|\zeta|^2 - i{\cal R}\zeta_1 ) \mathcal V(t,\zeta) 
=O_\omega(t)\left(
\mathbb I_3-\frac{(O_\omega(t)^\top\zeta) \otimes (O_\omega(t)^\top\zeta)}{|\zeta|^2}
\right)
\widehat f(O_\omega(t)^\top\zeta).
$$
By Duhamel's principle, we deduce
$$
\displaystyle \mathcal V(t,\zeta)= \displaystyle \int_{-\infty}^t e^{(s-t)(|\zeta|^2-i{\cal R}\zeta_1)}O_\omega(s)
\left(
\mathbb I_3-\frac{(O_\omega(s)^\top\zeta)\otimes (O_\omega(s)^\top\zeta)}{|\zeta|^2}
\right)
\widehat f(O_\omega(s)^\top\zeta)\,ds
$$
and using \eqref{vtoV}, we recover \eqref{formula}.

Let $(v_0,p_0)\in {\cal S}^\prime(\mathbb R^3)$ be a solution to 
\eqref{linear-whole} with $f=0$, then we see that
$\mbox{supp $\widehat v_0$}\subset\{0\}$.
In fact, since
\[
|\zeta|^2\widehat v_0+i\zeta\widehat p_0
-i{\cal R}\zeta_1\widehat v_0-|\omega|
\left[(e_1\times\zeta)\cdot\nabla_\zeta\widehat v_0-e_1\times\widehat v_0\right]=0,
\qquad i\zeta\cdot\widehat v_0=0,
\]
we have
$|\zeta|^2\widehat p_0=0$, which implies that
$\mbox{supp $\widehat p_0$}\subset\{0\}$ and that
\[
|\zeta|^2
\left[
|\zeta|^2\widehat v_0-i{\cal R}\zeta_1\widehat v_0-|\omega|
\left[(e_1\times\zeta)\cdot\nabla_\zeta\widehat v_0-e_1\times\widehat v_0\right]
\right]=0.
\]
Given arbitrary vector field 
$\varphi\in C_0^\infty(\mathbb R^3_\zeta\setminus\{0\})$, we set
\[
\tau(\zeta)=\int_0^\infty
e^{-(|\zeta|^2+i{\cal R}\zeta_1)t}
O_\omega(t)\varphi\left(O_\omega(t)^\top\zeta\right)\,dt
\in C_0^\infty(\mathbb R^3_\zeta\setminus\{0\}),
\]
which solves the adjoint system 
\[
|\zeta|^2\tau+i{\cal R}\zeta_1\tau+|\omega|
\left[(e_1\times\zeta)\cdot\nabla_\zeta\tau-e_1\times\tau\right]=\varphi.
\]
We thus obtain
\begin{equation*}
\begin{split}
\langle\widehat v_0, \varphi\rangle
&=\langle\widehat v_0,
|\zeta|^2\tau+i{\cal R}\zeta_1\tau+|\omega|
\left[(e_1\times\zeta)\cdot\nabla_\zeta\tau-e_1\times\tau\right]
\rangle  \\
&=\Big\langle
|\zeta|^2
\left\{
|\zeta|^2\widehat v_0-i{\cal R}\zeta_1\widehat v_0-|\omega|
\left[(e_1\times\zeta)\cdot\nabla_\zeta\widehat v_0-e_1\times\widehat v_0\right]
\right\},
\frac{\tau}{|\zeta|^2}
\Big\rangle=0,
\end{split}
\end{equation*}
yielding $\mbox{supp $\widehat v_0$}\subset\{0\}$.
Therefore, $v={\cal F}^{-1}\widehat v$ with \eqref{formula}
is the only solution to \eqref{linear-whole}
up to (specific) polynomials within ${\cal S}^\prime(\mathbb R^3)$.
It is actually the only solution when the polynomials are excluded
on account of the asymptotic behavior at infinity.
The same thing for the case $\omega=0$ is shown even more straightforward.
Thus the following $L^2$-estimate for 
\eqref{formula2}--\eqref{formula} plays an
important role.
\begin{lemma}
Suppose that $f$ is of the form $f=g+\mbox{\rm div } G$ with
\begin{equation}
[y \mapsto (1+|y|)G(y)] \in L^2(\mathbb{R}^3), \qquad
g\in L^1(\mathbb R^3), \qquad
[y \mapsto |y|g(y)] \in L^s(\mathbb R^3)
\label{ext-force-whole}
\end{equation}
for some $s\in [1,6/5)$.
\begin{enumerate}
\item
Let $\omega=0$, and
let $\widehat v$ be as in \eqref{formula2}.
If
\begin{equation}
\int_{\mathbb R^3}g(y)\,dy=0,
\label{null}
\end{equation}
then we have $v={\cal F}^{-1}\widehat v\in L^2(\mathbb R^3)$ subject to
\begin{equation}
\|v\|_{2,\mathbb R^3}
\leq C\left(\|g\|_{1,\mathbb R^3}+\|G\|_{2,\mathbb R^3}\right)
+C'\big\||y|G\big\|_{2,\mathbb R^3}+C''\big\||y|g\big\|_{s,\mathbb R^3}
\label{L2-est2-whole}
\end{equation}
with some constants $C,\,C',\,C''>0$ which are independent of
$g,G$ and $\xi$.

\item
Let $\omega\in\mathbb R^3\setminus\{0\}$, and
let $\widehat v$ be as in \eqref{formula}.
If
\begin{equation}
 e_1\cdot\int_{\mathbb R^3}g(y)\,dy=0,
\label{part-null}
\end{equation}
then we have $v={\cal F}^{-1}\widehat v\in L^2(\mathbb R^3)$ subject to
\begin{equation}
\begin{split}
&\|v\|_{2,\mathbb R^3} \leq C\left[\left(|\omega|^{-1/4}+|{\cal R}|^{1/2}|\omega|^{-1/2}\right)\left|\int_{\mathbb R^3}g(y)\,dy\right|+\|g\|_{1,\mathbb R^3}+\|G\|_{2,\mathbb R^3}\right]  \\
&\qquad \qquad \qquad 
+C'\big\||y|G\big\|_{2,\mathbb R^3}+C''\big\||y|g\big\|_{s,\mathbb R^3}
\end{split}
\label{L2-est-whole}
\end{equation}
with some constants $C,\,C',C''>0$ which are independent of
$g,G,\omega$ and ${\cal R}$.

\end{enumerate}
\label{square-est-whole}
\end{lemma}
\begin{proof}
Let us discuss the case $\omega\in\mathbb R^3\setminus\{0\}$ and consider
\[
\int_{\mathbb R^3}|\widehat v(\zeta)|^2d\zeta
=\int_{|\zeta|\geq 1} |\widehat v(\zeta)|^2d\zeta +\int_{|\zeta|<1} |\widehat v(\zeta)|^2d\zeta.
\]
Using the Schwarz inequality in the integral with respect to $t$ and then the Fubini theorem followed by the change of variable $\zeta \mapsto O_\omega(t)^\top\zeta$ together with the decomposition $\widehat f (\zeta) = \widehat g  (\zeta) + i \zeta \cdot \widehat G  (\zeta)$, 
 we see that the high frequency part of \eqref{formula} is estimated as
\begin{equation}
\begin{split}
\int_{|\zeta|\geq 1} |\widehat v(\zeta)|^2d\zeta
&\leq C\int_{|\zeta|\geq 1}
\left(\int_0^\infty e^{-|\zeta|^2t}
|\widehat f(O_\omega(t)\zeta)|\, dt \right)^2d\zeta  \\
& \leq C \int_{|\zeta|\geq 1}
\int_0^\infty e^{-|\zeta|^2t}
|\widehat f(O_\omega(t)\zeta)|^2\, dt\, \frac{d\zeta}{|\zeta|^2} 
\\
& = C \int_0^\infty  \int_{|\zeta|\geq 1}
e^{-|\zeta|^2t}
|\widehat f(\zeta)|^2\frac{d\zeta}{|\zeta|^2} \, dt  
\\
&=C\int_{|\zeta|\geq 1}\frac{|\widehat f(\zeta)|^2}{|\zeta|^4}\,d\zeta  \\
&\leq C\|\widehat g\|_{\infty,\mathbb R^3}^2
\int_{|\zeta|\geq 1}\frac{d\zeta}{|\zeta|^4}
+C\|\widehat G\|_{2,\mathbb R^3}^2 \\
&\leq C\left(\|g\|_{1,\mathbb R^3}^2+\|G\|_{2,\mathbb R^3}^2\right).
\end{split}
\label{high}
\end{equation}
Our main task is thus to study the low frequency part, which will be based on the decomposition
$$
\widehat f(O_\omega(t)\zeta) = \displaystyle  \widehat g(0)
+\int_0^1(O_\omega(t)\zeta)\cdot(\nabla\widehat g)(\sigma O_\omega(t)\zeta)d\sigma  + i (O_\omega(t)\zeta) \cdot \widehat G(O_\omega(t)\zeta).
$$
Note that the function $\widehat g$ is uniformly continuous
by $g\in L^1(\mathbb R^3)$ and thus $\widehat g(0)$ makes sense.
The above decomposition of $\widehat f(O_\omega(t)\zeta)$ splits \eqref{formula} into three parts:
\begin{equation*}
\begin{split}
&\qquad \int_{|\zeta|<1}    |\widehat v(\zeta)|^2d\zeta    \\
& =   \int_{|\zeta|<1}\Big|
\int_0^\infty e^{-(|\zeta|^2-i{\cal R}\zeta_1)t}
O_\omega(t)^\top
\left(
\mathbb I_3-\frac{(O_\omega(t)\zeta)\otimes (O_\omega(t)\zeta)}{|\zeta|^2}
\right) dt  \;\widehat g(0) \\
&\qquad \qquad 
+ \int_0^\infty e^{-(|\zeta|^2-i{\cal R}\zeta_1)t}
O_\omega(t)^\top
\left(
\mathbb I_3-\frac{(O_\omega(t)\zeta)\otimes (O_\omega(t)\zeta)}{|\zeta|^2}
\right) \int_0^1(O_\omega(t)\zeta)\cdot(\nabla\widehat g)(\sigma O_\omega(t)\zeta)d\sigma  dt \\
&\qquad \qquad + i \int_0^\infty e^{-(|\zeta|^2-i{\cal R}\zeta_1)t}
O_\omega(t)^\top
\left(
\mathbb I_3-\frac{(O_\omega(t)\zeta)\otimes (O_\omega(t)\zeta)}{|\zeta|^2}
\right) (O_\omega(t)\zeta)\cdot\widehat G(O_\omega(t)\zeta)
dt
\Big|^2 d\zeta \\
&\leq (I_1+I_2+I_3)^2.
\end{split}
\end{equation*}
For the last two integrals $I_2$ and $I_3$, one may ignore the oscillation.
In fact, we see from the Hardy inequality that
\begin{equation}
\begin{split}
I_3^2
&\leq C\int_{|\zeta|<1}\int_0^\infty
e^{-|\zeta|^2t}|\widehat G(O_\omega(t)\zeta)|^2dt\,d\zeta  = C\int_{|\zeta|<1}|\widehat G(\zeta)|^2\frac{d\zeta}{|\zeta|^2} \\
& \leq C\int_{\mathbb R^3}\frac{|\widehat G(\zeta)|^2}{|\zeta|^2} d\zeta
\leq C\|\nabla \widehat G\|_{2,\mathbb R^3}^2 
=C\big\||y| G\big\|_{2,\mathbb R^3}^2
\end{split}
\label{low-1}
\end{equation}
and that
\begin{equation}
\begin{split}
I_2^2
&\leq C\int_{|\zeta|<1}\int_0^\infty
e^{-|\zeta|^2t}
\int_0^1 |(\nabla\widehat g)(\sigma O_\omega(t)\zeta)|^2
d\sigma\,dt\,d\zeta  \\
&=C\int_{|\zeta|<1}
\int_0^1|(\nabla\widehat g)(\sigma\zeta)|^2d\sigma\,\frac{d\zeta}{|\zeta|^2} =C\int_0^1\int_{|\zeta|<\sigma}|(\nabla\widehat g)(\zeta)|^2\,
\frac{d\zeta}{|\zeta|^2}\,\frac{d\sigma}{\sigma} \\
&\leq C\int_0^1\left(\int_{|\zeta|<\sigma}|(\nabla\widehat g)(\zeta)|^{s/(s-1)}d\zeta\right)^{2(1-1/s)} \sigma^{-6(1-1/s)}\,d\sigma  \\
&\leq C\|\nabla\widehat g\|_{s/(s-1),\mathbb R^3}^2 \leq C\big\||y|g\big\|_{s,\mathbb R^3}^2
\end{split}
\label{low-2}
\end{equation}
where \eqref{ext-force-whole} with $s\in [1,6/5)$ is employed.
We thus obtain
\begin{equation}
I_2\leq C\big\||y|g\big\|_{s,\mathbb R^3}, \qquad
I_3\leq C\big\||y| G\big\|_{2,\mathbb R^3}.
\label{osc-ignore}
\end{equation}

For the crucial part $I_1$ we do need the assumption
\eqref{part-null}, that is, $e_1\cdot  \widehat g(0)=0$,
as well as the oscillation caused by the rotation.
By the relations
$$
 \widehat g(0)
=(e_1\cdot \widehat g(0))e_1+(e_1\times  \widehat g(0))\times e_1 = (e_1\times   \widehat g(0))\times e_1
$$
and
$$
O_\omega(t)^\top
\left(
\mathbb I_3-\frac{(O_\omega(t)\zeta)\otimes (O_\omega(t)\zeta)}{|\zeta|^2}
\right)
O_\omega(t) = \mathbb I_3-\frac{\zeta\otimes\zeta}{|\zeta|^2}
$$
we may write
$$
I_1^2=\int_{|\zeta|<1}|w(\zeta)|^2\, d\zeta
$$
with
$$
\begin{array}{rcl}
w(\zeta)&:=&\displaystyle \int_0^\infty e^{-(|\zeta|^2-i{\cal R}\zeta_1)t}
\left(
\mathbb I_3-\frac{\zeta \otimes \zeta}{|\zeta|^2}
\right) O_\omega(t)^\top dt\,
[(e_1\times  \widehat g(0))\times e_1] \medskip  \\
&=&\displaystyle \left(
\mathbb I_3-\frac{\zeta\otimes\zeta}{|\zeta|^2}
\right)
\phi(\zeta), 
\end{array}
$$
where
\[
\phi(\zeta)
:=\int_0^\infty e^{-(|\zeta|^2-i{\cal R}\zeta_1)t}O_\omega(t)^\top dt\,
[(e_1\times  \widehat g(0))\times e_1].
\]
An elementary computation yields
\[
\phi(\zeta)
=\frac{1}{2}
\left[
\begin{array}{c}
0 \\
\alpha_- J_+ +\alpha_+ J_- \\
i\alpha_- J_+ -i\alpha_+ J_-
\end{array}
\right]
\]
with
\[
\alpha_\pm:=\widehat g_2(0)\pm i\widehat g_3(0), \qquad
J_\pm(\zeta):=\frac{1}{|\zeta|^2-i({\cal R}\zeta_1\pm |\omega|)}.
\]
We then find
\begin{equation}
\int_{|\zeta|<1}|J_\pm(\zeta)|^2 d\zeta\leq
\int_{\mathbb R^3}\frac{d\zeta}{|\zeta|^4+({\cal R}\zeta_1\pm |\omega|)^2}
\leq C\left(
\frac{1}{\sqrt{|\omega|}}+\frac{|\cal R|}{|\omega|}
\right)
\label{dependence}
\end{equation}
for all $|\omega|>0$ and ${\cal R}\in\mathbb R$, see \eqref{parame}.
In fact, we immediately see that
\[
\int_{\mathbb R^3}\frac{d\zeta}{|\zeta|^4+|\omega|^2}
=\frac{1}{\sqrt{|\omega|}}
\int_{\mathbb R^3}\frac{d\zeta}{|\zeta|^4+1}
\]
for the case ${\cal R}=0$,
while the integral above for the other case ${\cal R}\neq 0$ is rewritten as
\[
\frac{2\pi}{\sqrt{|\omega|}}\int_0^\pi\sin\theta\int_0^\infty
\frac{\rho^2\,d\rho}{\rho^4+\Big(\frac{\cal R}{\sqrt{|\omega|}}\,\rho\cos\theta\pm 1\Big)^2}\,d\theta.
\]
For the latter case, it is reasonable to split the
integral with respect to $\rho$ into two parts:
$$
\int_0^\infty=
\int_0^{\sqrt{|\omega|}/(2|{\cal R}|)}
+\int_{\sqrt{|\omega|}/(2|{\cal R}|)}^\infty.
$$
Then we have
\[
\int_0^{\sqrt{|\omega|}/(2|{\cal R}|)}\frac{\rho^2\,d\rho}{\rho^4+\Big(\frac{\cal R}{\sqrt{|\omega|}}\,\rho\cos\theta\pm 1\Big)^2}
\leq \int_0^\infty\frac{\rho^2}{\rho^4+\frac{1}{4}}\,d\rho
\]
as well as
\[
\int_{\sqrt{|\omega|}/(2|{\cal R}|)}^\infty \frac{\rho^2\,d\rho}{\rho^4+\Big(\frac{\cal R}{\sqrt{|\omega|}}\,\rho\cos\theta\pm 1\Big)^2}
\leq \int_{\sqrt{|\omega|}/(2|{\cal R}|)}^\infty\,\frac{d\rho}{\rho^2}
\leq\frac{2|{\cal R}|}{\sqrt{|\omega|}}.
\]
We thus obtain \eqref{dependence}, which leads us to
\begin{equation}
I_1\leq
C\left(
\frac{1}{\sqrt{|\omega|}}+\frac{|\cal R|}{|\omega|}
\right)^{1/2}| \widehat g(0)|.
\label{osc}
\end{equation}
We collect \eqref{high}, \eqref{osc-ignore} and \eqref{osc} to conclude
\eqref{L2-est-whole}.

For the other case $\omega=0$, the high frequency part is the same
as in $\mbox{\eqref{high}}_3$ and the low frequency part can be treated
as in \eqref{low-1}--\eqref{low-2} by use of $\widehat g(0)=0$,
so that we obtain \eqref{L2-est2-whole}.
The proof is complete.
\end{proof}

\begin{proof}[Proof of Proposition \ref{square-est}.]
We first discuss the flow with vanishing flux condition.
Let us fix $R>0$ such that $\mathbb R^3\setminus\Omega\subset B_R$,
and take a cut-off function
$\psi\in C_0^\infty(B_{3R}; [0,1])$ such that
$\psi(x)=1$ for $x\in B_{2R}$.
Given $(u,q)$ which is of class \eqref{class} together with
$\int_{\partial\Omega}n\cdot u\,d\gamma=0$
and satisfies \eqref{linear-ext},
we set
\[
\widetilde u=(1-\psi)u+\mathbb B[u\cdot\nabla\psi], \qquad
\widetilde q=(1-\psi)q,
\]
where $\mathbb B$ denotes the Bogovskii operator in the domain
$A_{R,3R}$, see \cite{Bog} and \cite[Theorem III.3.3, p.179]{G}.
Note that $u\cdot\nabla\psi \in W_0^{1,2}(A_{R,3R})$ with 
\[
\int_{A_{R,3R}} u\cdot\nabla\psi  dx =  
\int_{\partial B_{R}} u\cdot \frac{-x}{R} d\gamma =
\int_{\partial \Omega} u\cdot n d\gamma = 0
\]
so that $\mathbb B[u\cdot\nabla\psi]\in W^{2,2}_0(A_{R,3R})$ is well-defined with the relation $\div\mathbb B[u\cdot\nabla\psi]=u\cdot\nabla\psi$ and

\begin{equation}
\|  \mathbb B[u\cdot\nabla\psi] \|_{2,2,A_{R,3R}} \leq C \| u\cdot\nabla\psi \|_{1,2,A_{R,3R}}
\leq C \| u  \|_{1,2,A_{R,3R}} \leq C(\|u\|_{6,A_{R,3R}}+\|\nabla u\|_{2,A_{R,3R}})
\leq C \| \nabla u  \|_{2,\Omega}.
\label{BogEst}
\end{equation}
Then the pair $(\widetilde u, \widetilde q)$ obeys
\begin{equation}
-\Delta\widetilde u+\nabla\widetilde q-V\cdot\nabla\widetilde u
+\omega\times\widetilde u
=h+\mbox{div $[(1-\psi)F]$}, \quad
\div\widetilde u=0 \quad \mbox{ in } \mathbb R^3,
\label{reduce-1}
\end{equation}
where
\begin{equation*}
\begin{split}
h&=
2\nabla\psi\cdot\nabla u+(\Delta\psi+V\cdot\nabla\psi)u
-\Delta\mathbb B[u\cdot\nabla\psi]-V\cdot\nabla\mathbb B[u\cdot\nabla\psi] \\
&\quad +\omega\times\mathbb B[u\cdot\nabla\psi]-(\nabla\psi)q
+F(\nabla\psi)
\end{split}
\end{equation*}
which satisfies
\begin{equation}
\int_{\mathbb R^3}h(x)\,dx=N,
\label{force}
\end{equation}
and $N$ denotes the net force \eqref{net-force}.
In fact,
\begin{equation*}
\begin{split}
\int_{\mathbb R^3}h(x)\,dx
&=-\int_{A_{R,3R}} \mbox{div $[\sigma(\widetilde u,\widetilde q)+\widetilde u\otimes V-(\omega\times x)\otimes\widetilde u+(1-\psi)F]$}\,dx
 \\
&=-\int_{|x|=3R}[\sigma(u,q)+u\otimes V-(\omega\times x)\otimes u+F]\frac{x}{3R}\,d\gamma
\end{split}
\end{equation*}
yielding \eqref{force} by use of \eqref{linear-ext}.

By the transformations \eqref{changevariables}, we consider
\begin{equation*}
\begin{split}
&v(y)=M\widetilde u\left(M^\top y+\frac{\omega\times\xi}{|\omega|^2}\right), \\
&p(y)=\widetilde q\left(M^\top y+\frac{\omega\times\xi}{|\omega|^2}\right), \\
&g(y)=Mh\left(M^\top y+\frac{\omega\times\xi}{|\omega|^2}\right), \\
&G(y)=\left(M[(1-\psi)F]M^\top\right)\left(M^\top y+\frac{\omega\times\xi}{|\omega|^2}\right),
\end{split}
\end{equation*}
and it can be shown that $(v,p)$ obeys \eqref{linear-whole} with $f=g+\mbox{div }G$
which satisfies \eqref{ext-force-whole}--\eqref{part-null} subject to
\begin{equation}
\begin{split}
&\quad \|g\|_{1,\mathbb R^3}+\big\||y|g\big\|_{1,\mathbb R^3} \\
&=\|h\|_{1,\mathbb R^3}+
\left\|\Big|x-\frac{\omega\times\xi}{|\omega|^2}\Big| h\right\|_{1,\mathbb R^3}  \\
&\leq C_R
\left(1+\frac{|\omega\times\xi|}{|\omega|^2}\right)
\big[
(1+|\xi|+|\omega|)\|\nabla u\|_{2,\Omega}+\|q\|_{2,\Omega}+\|F\|_{2,\Omega}
\big],
\end{split}
\label{force-est1}
\end{equation}
as well as
\begin{equation}
\|G\|_{2,\mathbb R^3} + \||y|G\|_{2,\mathbb R^3} \leq \left(1 + \frac{|\omega\times\xi|}{|\omega|^2}\right) \|F\|_{2,\Omega} + \| |x|F\|_{2,\Omega}.
\label{force-est2}
\end{equation}
The condition \eqref{part-null} is in fact verified as
\begin{equation}
e_1\cdot\int_{\mathbb R^3}g(y)\,dy
=\frac{\omega}{|\omega|}\cdot\int_{\mathbb R^3}h(x)\,dx=0
\label{reduced-null}
\end{equation}
by use of \eqref{force} and the assumption $\omega\cdot N=0$.

Since $u\in L^6(\Omega)$ implies that
$v\in L^6(\mathbb R^3)\subset {\cal S}^\prime(\mathbb R^3)$,
$v$ coincides with \eqref{formula} on the Fourier side by
the reasoning mentioned just before Lemma \ref{square-est-whole}.
By taking \eqref{force-est1}--\eqref{reduced-null} into account,
we obtain from Lemma \ref{square-est-whole} that
there are constants $C,\,C'>0$ satisfying
\[
\|u\|_{2,\Omega{3R}}
\leq\|\widetilde u\|_{2,\mathbb R^3}
=\|v\|_{2,\mathbb R^3}
\leq C K(u,q,F,\xi,\omega,0)+C'\||x|F\|_{2,\Omega},
\]
where $K(u,q,F,\xi,\omega,0)$ is given by \eqref{quantity} with $\Phi=0$.
Here, $\left|\int_{\mathbb R^3}g(y)dy\right|$
has been just replaced by $\|g\|_{1,\mathbb R^3}$ in \eqref{L2-est-whole}.
On the other hand, we have
\[
\|u\|_{2,\Omega_{3R}}
\leq C\|u\|_{6,\Omega}
\leq C\|\nabla u\|_{2,\Omega}.
\]
Combining the estimates above implies \eqref{L2-est} when $\Phi=0$.
If we prefer to keep
\[
\left|\int_{\mathbb R^3}g(y)\,dy\right|
=\left|\int_{\mathbb R^3}Mh(x)\,dx\right|
=|N|
\]
in \eqref{L2-est-whole} as it is, we obtain \eqref{L2-est3} with
$\Phi=0$ as well.

For general case without any condition at the boundary $\partial\Omega$,
let us reduce the problem to the case discussed above by
lifting the flux $\Phi=\int_{\partial\Omega}n\cdot u\,d\gamma$.
We fix $x_0\in\mbox{int $(\mathbb R^3\setminus\Omega)$}$
and take the flux carrier
\[
{\cal W}(x)=\Phi\nabla\frac{1}{4\pi |x-x_0|}.
\]
Note that one cannot always choose $x_0=0$ (center of mass of the rigid body).
Then the pair
\[
U=u-{\cal W},  \qquad
Q=q-(\xi+\omega\times x_0)\cdot {\cal W}
\]
obeys
\[
-\Delta U+\nabla Q-V\cdot\nabla U+\omega\times U=\mbox{div $F$}, \qquad
\mbox{div $U$}=0
\]
in $\Omega$ subject to
$\int_{\partial\Omega}n\cdot U\,d\gamma=0$.
Concerning the net force \eqref{net-force}, we observe
\[
N_{U,Q}=N_{u,q}+\Phi(\omega\times x_0)
\]
as verified in \cite[Section 6]{HST}
(in which the nonlinear momentum flux is discussed, however,
all the computations for the linear part are included there),
so that the condition $\omega\cdot N_{u,q}=0$ implies
$\omega\cdot N_{U,Q}=0$.
Hence, we already know that
\[
\|U\|_{2,\Omega}
\leq C K(U,Q,F,\xi,\omega,0)+C'\||x|F\|_{2,\Omega}.
\]
Since
\[
\|\nabla U\|_{2,\Omega}\leq\|\nabla u\|_{2,\Omega}+C|\Phi|, \qquad
\|Q\|_{2,\Omega}\leq \|q\|_{2,\Omega}+C(|\xi|+|\omega|)|\Phi|
\]
as well as
\[
\|{\cal W}\|_{2,\Omega}\leq C|\Phi|,
\]
one conculdes \eqref{L2-est}.
The other case $\omega=0$ is also discussed in the same way as above.

\end{proof}

\section{The state system and the cost functional}\label{sec_pb}

As explained in the introduction, our aim is to find a control $v_*$ for which 
\begin{equation}
\inf W(v,p) = \inf\int_{\partial\Omega} v \cdot \sigma(v,p) n\,d\gamma,
\label{infdrag}
\end{equation}
is attained (see \eqref{workdrag}). In \eqref{infdrag} the infimum is taken over the set of all possible states $(v,p)$ satisfying \eqref{000}--\eqref{004} for $v_*$ either in 
\begin{equation}
\mathcal{V}_\tau := \left\{ v_* \in W^{3/2,2}(\partial \Omega) \ ;\  v_*\cdot n = 0 \mbox{ on } \partial \Omega\right\}
\label{defesVtau}
\end{equation}
or in
\begin{equation}
\mathcal{V}_{\Gamma} := \left\{ v_* \in W^{3/2,2}(\partial \Omega) \ ;\  v_*=0\quad \text{on} \ \partial \Omega\setminus \Gamma \right\},
\label{defesVchi}
\end{equation}
where $\Gamma$ is a nonempty open subset of $\partial \Omega$.

We recall that in \cite{HST} we studied the case of subspaces of $\mathcal{V}_\tau$ and of $\mathcal{V}_{\Gamma}$ of finite dimension:
\begin{equation}
{\mathcal C}_\tau := \operatorname{span} \left\{ (g^{(i)}\times n) \times n, (G^{(i)} \times n) \times n \ ; \ i=1,2,3 \right\},
\label{controlspacetau}
\end{equation}
and
\begin{equation}
{\mathcal C}_\chi := \operatorname{span}\left\{ \chi g^{(i)}, \chi G^{(i)}\ ; \ i=1,2,3 \right\},
\label{controlspace}
\end{equation} 
where $\chi \not \equiv 0$ is a non-negative smooth function such that $\supp(\chi)\subset \Gamma$
and
where $g^{(i)}$, $G^{(i)}$ are defined as follows. 

First, we introduce a set of  generalized Oseen systems associated with the basic rigid motions: for each $i \in \{1,2,3\}$, $(v^{(i)},q^{(i)})$ and $(V^{(i)},Q^{(i)})$ are the solutions of 
\begin{equation}
\begin{array}{c}
\displaystyle -\div \sigma(v^{(i)},q^{(i)}) +  (\xi+\omega\times x)\cdot \nabla v^{(i)} -  \omega \times v^{(i)} =0 \quad \text{in } \Omega \medskip \\
\displaystyle \div v^{(i)} =0 \quad \text{in } \Omega \medskip \\
\displaystyle v^{(i)}=e_i \quad \text{on}\ \partial \Omega \medskip \\
\displaystyle \lim_{|x|\to \infty} v^{(i)}= 0 
\end{array}
\label{313}
\end{equation}
and
\begin{equation}
\begin{array}{c}
\displaystyle
-\div \sigma(V^{(i)},Q^{(i)})+  (\xi+\omega\times x)\cdot \nabla V^{(i)} -  \omega \times V^{(i)} =0 \quad \text{in } \Omega \medskip \\
\displaystyle \div V^{(i)} =0 \quad \text{in } \Omega \medskip \\
\displaystyle V^{(i)}=e_i\times x \quad \text{on}\ \partial \Omega \medskip \\
\displaystyle \lim_{|x|\to \infty} V^{(i)}= 0
\end{array}
\label{323}
\end{equation}
where $(e_1,e_2,e_3)$ is the canonical basis of $\mathbb{R}^3$. 
Then the fields $g^{(i)}$ and $G^{(i)}$ are given by
\begin{equation}\label{014}
g^{(i)}:=\sigma(v^{(i)},q^{(i)})n  \quad \text{on } \partial \Omega,
\end{equation}
\begin{equation}\label{024}
G^{(i)}:=\sigma(V^{(i)},Q^{(i)})n \quad \text{on } \partial \Omega,
\end{equation}
being in $W^{3/2,2}(\partial\Omega)$ provided $\partial\Omega\in C^3$, see \cite[Lemma 3.2]{HST}.
If $|\xi|$ and $|\omega|$ are small enough, then ${\mathcal C}_\tau$ and ${\mathcal C}_\chi$ are of dimension 6 (see \cite[Theorem 1.1]{HST}). We proved in  \cite{HST} that, for $|\xi|$ and $|\omega|$ small enough, there exists only one $v_*\in {\mathcal C}_\tau$ (resp. $v_*\in {\mathcal C}_\chi$) such that there exists a solution $(v,p)$ of \eqref{000}--\eqref{004}. 

Here, we want to consider controls $v_*$ in $\mathcal{V}_\tau$ or $\mathcal{V}_{\Gamma}$ and characterize the optimal controls that minimize \eqref{workdrag}. However, before starting the analysis of the minimization problem, we must notice that for an arbitrary $v_*\in \mathcal{V}_\tau$ or $v_*\in \mathcal{V}_{\Gamma}$, 
system \eqref{000}--\eqref{004} has no solution in general since the unique solution $(v,p)$ of \eqref{000}--\eqref{003} 
may not verify the self-propelled conditions \eqref{903}, \eqref{004}. 
In order to handle this difficulty, the solution to our problem consists of a boundary velocity $v_*$ which can be decomposed into two parts, one part which effectively acts as the infinite dimensional control (to alleviate the presentation, we keep denoting it by $v_*$)
and another part, say $v_*^{\mathcal C}$, that ``corrects'' the control in order to enforce the self-propelled conditions and belongs to the finite dimensional spaces ${\mathcal C}_\tau$ and ${\mathcal C}_\chi$.


Then, assuming that the rigid body velocity $V$ is given and having in mind the boundary control problems for $v_* \in  {\mathcal V}_\tau$ or $v_* \in  \mathcal{V}_{\Gamma}$, we write the state system \eqref{000}--\eqref{004} in the form
\begin{gather}
-\div \sigma(v,p)+  (v-V)\cdot \nabla v+ \omega\times v =0 \quad \text{in} \ \Omega \label{ncy0.0}\\
\div v =0 \quad \text{in} \ \Omega \label{ncy0.1}\\
v=V+\vstar+\vperp \quad \text{on}\ \partial \Omega\label{ncy0.2}\\
\lim_{|x|\to \infty} v(x) = 0\label{ncy0.3}\\
m \xi\times \omega + \int_{\partial \Omega} \left[-\sigma(v,p)n +  \left((v_*+v_*^{\mathcal C})\cdot n \right) (v_* +v_*^{\mathcal C}+ V + \omega \times x)
  \right] \ d \gamma =0
\label{ncy0.4}\\
(I \omega)\times \omega + \int_{\partial \Omega} x\times \left[-\sigma(v,p)n +  \left((v_*+v_*^{\mathcal C})\cdot n \right) (v_* +v_*^{\mathcal C}+ V + \omega \times x) \right] \ d \gamma =0
\label{ncy0.5}
\\
v_*^{\mathcal C} \in {\mathcal C}_\tau \quad \text{or} \quad v_*^{\mathcal C} \in {\mathcal C}_\chi.
\label{ncy0.51}
\end{gather}

In what follows, we use the notation
\begin{equation}\label{weight}
\varpi (x):= 
\left\{ \begin{array}{ll}
\displaystyle \left(1+\left|x-\frac{\omega\times\xi}{|\omega|^2}\right|\right)
\left(1 + 2\,\frac{|\omega\cdot \xi|}{| \omega |}\, s(x)\right),
\qquad & \omega \not=0, \medskip \\
\displaystyle (1+|x|)\left(1+2 (|\xi||x|+\xi\cdot x )\right),  & \omega = 0,
\end{array} 
\right.
\end{equation}
where
\[
s(x):=\left|x-\frac{\omega\times\xi}{|\omega|^2}\right|
+\frac{\operatorname{sign}(\omega\cdot\xi)}{|\omega|}\omega\cdot x, \qquad
\omega\neq 0.
\]

Extending the results of \cite[Theorem 1.1 and Theorem 1.2]{HST}, we obtain the following for our state system.
\begin{theorem}\label{Tmain} 
Let $\Omega$ be of class $C^3$. There exist constants $c_0, C_1, C_2 >0$, which depend on $\Omega$, such that if $\xi, \omega\in\mathbb R^3$ and $\vstar\in \mathcal{V}_\tau$
(resp. $\vstar\in \mathcal{V}_{\Gamma}$)
satisfy
\begin{equation}\label{new010}
|\xi|\leq c_0, \qquad |\omega|\leq c_0,
\qquad 
\|\vstar\|_{3/2,2,\partial\Omega}\leq c_0,
\end{equation}
then the following assertions hold.
\begin{enumerate}
\item
A solution $(v,p,\vperp)$ of the problem \eqref{ncy0.0}--\eqref{ncy0.51} can be found within the class
\begin{equation}\label{new020}
\varpi v\in L^\infty(\Omega), \quad
(v,p) \in W^{2,2}(\Omega)\times W^{1,2}(\Omega), \quad
\vperp \in \mathcal{C}_\tau \quad (\text{resp.}\  \vperp\in \mathcal{C}_\chi)
\end{equation}
along with estimates
\begin{equation}
\lceil v\rceil_{1,\varpi,\Omega}
+\|\nabla v\|_{2,\Omega}
+ \|\vperp \|_{3/2,2,\partial\Omega}
\leq C_1\big(|(\xi,\omega)|
+\|\vstar\|_{3/2,2,\partial\Omega}
\big),
\label{main-est}
\end{equation}
\begin{equation}
\|\nabla^2 v\|_{2,\Omega}
+\|p\|_{1,2,\Omega}
\leq C_2\big(|(\xi,\omega)|+\|\vstar\|_{3/2,2,\partial\Omega}\big),
\label{main-est-add}
\end{equation}
and the energy equation
\begin{equation}\label{drag}
\int_{\partial\Omega}\big(\sigma(v,p)n\big)\cdot v\,d\gamma
=2  \int_{\Omega}|D(v)|^2\,dx
+\frac{1}{2}\int_{\partial\Omega}(v_*+v_*^{\mathcal C})\cdot n |V+v_*+v_*^{\mathcal C}|^2\,d\gamma.
\end{equation}

\item 
The solution of the problem \eqref{ncy0.0}--\eqref{ncy0.51} is unique (up to constants for the pressure) within the class of functions satisfying \eqref{main-est} as well as $\vperp \in\mathcal{C}_\tau$ (resp. $\vperp\in \mathcal{C}_\chi$). The pressure is singled out under the additional condition $p\in L^2(\Omega)$.
\end{enumerate}
\end{theorem}
The emphasis is finite kinetic energy $v\in L^2(\Omega)$, see \eqref{new020}, as a consequence of the self-propelling condition \eqref{ncy0.4}, and this helps us to justify the energy relation \eqref{drag}.
As in \cite{HST}, our notion of solution to \eqref{ncy0.0}--\eqref{ncy0.51} is the standard weak solution $\nabla v\in L^2(\Omega)$, such that
\begin{equation}
2  \int_{\Omega} D(v):D(\varphi) dx +  \int_{\Omega} (v-V) \cdot \nabla v \cdot \varphi dx +  \int_{\Omega} (\omega \times v)\cdot \varphi dx = 0, \quad \forall \varphi \in {\mathcal D}(\Omega),
\label{weakform}
\end{equation}
which, due to the extra regularity obtained in Theorem \ref{Tmain}, will satisfy the equations \eqref{903}--\eqref{004} in the strong form. We will only sketch the proof of Theorem \ref{Tmain} in the case of localized controls $\mathcal{C}_\chi$ since it is completely similar to the proof of \cite[Theorem 1.1]{HST}.
\begin{proof}[Proof of Theorem \ref{Tmain}]
Let 
$$
\mathcal{X}:=\left\{
(v,\alpha,\beta)\in D^{1,2}(\Omega)\times \mathbb R^3\times \mathbb R^3 \ ; \ 
\lceil v\rceil_{1,\varpi,\Omega}<\infty \right\}
$$
endowed with the norm
$$
\|(v,\alpha,\beta)\|_{\mathcal{X}} 
:=|v|_{1,2,\Omega} + \lceil v\rceil_{1,\varpi,\Omega} +|(\alpha,\beta)|.
$$
Assume $(\overline v,\overline \alpha,\overline \beta)\in \mathcal{X}$, and let 
$$
\overline v_*^{\mathcal C}:=\sum_{i=1}^3 \Big(\overline \alpha_i \chi g^{(i)}+\overline \beta_i \chi G^{(i)}\Big),
$$
$$
f(\overline{v}):=- \overline{v}\cdot \nabla \overline{v}=\div F(\overline{v}), \quad F(\overline{v}):=-\overline{v}\otimes \overline{v},
$$
\begin{equation}\label{077}
\xi_f(\overline v_*^{\mathcal C}):=- \int_{\partial \Omega} ((v_*+\overline v_*^{\mathcal C}) \cdot n ) ( (v_*+\overline v_*^{\mathcal C}+ V + \omega \times x) d\gamma  
	-  m \xi \times \omega  
	- \int_{\partial \Omega} (V\cdot n) (V+v_*+\overline v_*^{\mathcal C}) d \gamma,
\end{equation}
\begin{equation}\label{078}
\omega_f(\overline v_*^{\mathcal C}):= - \int_{\partial \Omega} x\times ( {v_*} +\overline v_*^{\mathcal C}+ V + \omega \times x ) ((v_*+\overline v_*^{\mathcal C})\cdot n )d\gamma  
-  (I\omega) \times \omega - \int_{\partial \Omega} x \times (V+v_*+\overline v_*^{\mathcal C})(V\cdot n)d \gamma.
\end{equation}

Following \cite[Lemma 4.1]{HST}, we introduce the following auxiliary linear systems
\begin{equation}
\left\{
\begin{array}{c}
\displaystyle -\div \sigma(u^{(j)},p^{(j)}) - V \cdot \nabla u^{(j)} + \omega \times u^{(j)} =0 \medskip \\
\displaystyle \div u^{(j)} =0\medskip \\
\displaystyle u^{(j)}=\chi g^{(j)}  \mbox{ or }  (g^{(i)}\times n) \times n \quad \text{on}\ \partial \Omega\medskip\\
\displaystyle \lim_{|x|\to \infty} u^{(j)} = 0,
\end{array}
\label{113xi}
\right.
\end{equation}
\begin{equation}
\left\{
\begin{array}{c}
\displaystyle
-\div \sigma(U^{(j)},P^{(j)}) - V \cdot \nabla U^{(j)} + \omega \times U^{(j)}=0\medskip\\
\displaystyle \div U^{(j)} =0 \medskip \\
\displaystyle U^{(j)}=\chi G^{(j)} \mbox{ or } (G^{(i)}\times n) \times n \quad \text{on}\ \partial \Omega \medskip \\
\displaystyle \lim_{|x|\to \infty} U^{(j)} = 0.
\end{array}
\label{123xi}
\right.
\end{equation}
\begin{equation}
\left\{
\begin{array}{c}
\displaystyle
-\div \sigma(u_f,p_f) - V \cdot \nabla u_f +  \omega \times u_f =f(\overline{v})\medskip\\
\div u_f =0\medskip\\
u_f=V+v_* \quad \text{on}\ \partial \Omega\medskip\\
\displaystyle \lim_{|x|\to \infty} u_f = 0.
\end{array}
\label{133xi}
\right.
\end{equation}
Using systems \eqref{113xi}--\eqref{133xi} and \cite[Proposition 4.5]{HST}, we can solve the following problem: for any $\vstar\in \mathcal{V}_{\Gamma}$,
and for $\xi$ and $\omega$ satisfying \eqref{new010} with some constant $c_0$ small enough,
there exists a unique $(v,\vperp,p)$ such that
\begin{gather}
-\div \sigma(v,p)  -V\cdot \nabla v+ \omega\times v =f(\overline{v}) \quad \text{in} \ \Omega \label{lis2.0}\\
\div v =0 \quad \text{in} \ \Omega \label{ncy2.1}\\
v=V+\vstar+\vperp \quad \text{on}\ \partial \Omega\label{ncy2.2}\\
\lim_{|x|\to \infty} v(x) = 0\label{ncy2.3}\\
- \int_{\partial \Omega} [\sigma(v,p)n + (V\cdot n) v ] d \gamma =\xi_f(\overline v_*^{\mathcal C})
 \label{ncy2.4}\\
- \int_{\partial \Omega} x\times [\sigma(v,p)n + (V\cdot n) v ] \ d \gamma =\omega_f(\overline v_*^{\mathcal C})
\label{ncy2.5}\\
 v_*^{\mathcal C}=\sum_{i=1}^3  \Big(\alpha_i \chi g^{(i)}+ \beta_i \chi G^{(i)} \Big) \in \mathcal{C}_\chi.
\label{ncy2.99}
\end{gather}
Moreover, we have
\begin{multline}\label{loc-est}
\sup_{x\in\Omega}[\varpi(x)|v(x)|]
+\|\nabla v\|_{1,2,\Omega} +  \| p\|_{1,2,\Omega}
+|(\alpha,\beta) |
\\
\leq C\big[|(\xi_f(\overline v_*^{\mathcal C}),\omega_f(\overline v_*^{\mathcal C}))|
+\|f(\overline{v})\|_{2,\Omega}
+\lceil F(\overline{v}) \rceil_{2,\varpi,\Omega}
+| (\xi,\omega)|+\|\vstar\|_{3/2,2,\partial\Omega}
+| (\xi,\omega)|(|(\xi,\omega)|+\|\vstar\|_{3/2,2,\partial\Omega})
\big].
\end{multline}
This allows us to define the mapping
$$
\mathcal{Z} : \mathcal X\ni (\overline v,\overline \alpha,\overline \beta) \mapsto ( v, \alpha, \beta)\in \mathcal X,
$$ 
where $(v,v_*^{\mathcal C})$ is the solution of \eqref{lis2.0}--\eqref{ncy2.99}.

Following the proof of \cite[Theorem 1.1]{HST}, based on \cite[Proposition 4.5]{HST}, we obtain
$$
\| \mathcal{Z}(\overline v,\overline \alpha,\overline \beta)\|_{\mathcal{X}} \leq C
\big(|(\xi,\omega)|+\|\vstar\|_{3/2,2,\partial\Omega}+|(\xi,\omega)|^2+\|\vstar\|_{3/2,2,\partial\Omega}^2\big)
+C\|(\overline v,\overline \alpha,\overline \beta)\|_{\mathcal{X}}^2.
$$
Taking $c_0 \in (0,1)$ small enough in  \eqref{new010}, we see that a ball
\begin{equation}
\mathcal{X}_{R_0}:=\left\{ (v,\alpha,\beta)\in \mathcal{X} \ ; \|(v,\alpha,\beta)\|_{\mathcal{X}} \leq  R_0 \right\}, \qquad
 R_0=C_1\big(|(\xi,\omega)|+\|\vstar\|_{3/2,2,\partial\Omega}\big)
\label{inv-ball}
\end{equation}
is invariant by $\mathcal{Z}$.
In a similar way, we also obtain that $\mathcal{Z}$ is a strict contraction on $\mathcal{X}_{R_0}$ for $c_0$ small enough. This gives us the existence of a solution to problem \eqref{ncy0.0}--\eqref{ncy0.51}. The corresponding estimates \eqref{main-est}--\eqref{main-est-add} follow from \eqref{loc-est}.

To prove that $v\in L^2(\Omega)$, we apply \cite[Theorem 1.2]{HST} where we only have to replace $v_*$ by $v_*+v_*^{\mathcal C}$ (see the definition (1.13) of $N$ in \cite{HST}).

In order to obtain the energy equation, 
for $R>0$ large enough such that ${\mathbb R}^3\setminus \Omega \subset B_R$, we use a radially symmetric cut-off function
$\psi_R(x)=\widetilde\psi(|x|/R)$ with $\widetilde\psi\in C_0^1((-2,2);[0,1])$ which fulfills
$\widetilde\psi\equiv 1$ on $[-1,1]$.  
Then we have
\begin{equation}
\|\nabla\psi_R\|_{q,\mathbb R^3}=CR^{-1+3/q} \qquad (3\leq q\leq\infty)
\label{cut-est}
\end{equation}
as well as $(\omega \times x )\cdot \nabla \psi_R=0$. 
From those properties it follows that
$$
\|V\cdot \nabla \psi_R\|_{\infty,\mathbb R^3} = \|\xi \cdot \nabla \psi_R\|_{\infty,\mathbb R^3}
= C/R.
$$
Taking the inner product of both sides of the equation \eqref{ncy0.0} with
$\psi_R v$ and integrating by parts over $\Omega$ yield
\begin{multline*}
\displaystyle - \int_{\partial\Omega}\big(\sigma(v,p)n\big)\cdot v \,d\gamma
+\int_{\Omega} \sigma(v,p) : (v\otimes\nabla \psi_R) \ dx
+  2 \int_{\Omega} |D(v)|^2 \psi_R   \ dx \medskip \\
\displaystyle + \int_{\partial \Omega} \frac{|v|^2}{2} (v-V)\cdot n \ d\gamma
 -\int_{\Omega} \frac{|v|^2}{2} (v-V)\cdot \nabla \psi_R \ dx = 0
\end{multline*}
Since $v\in W^{1,2}(\Omega)$ and $p \in L^2(\Omega)$, we have
$$
\left|-\int_{\Omega} \frac{|v|^2}{2} (v-V)\cdot \nabla \psi_R \ dx
+\int_{\Omega} \sigma(v,p) : (v\otimes\nabla \psi_R) \ dx\right|
\leq \frac{C}{R} \int_{\Omega}\left( |v|^3+|v|^2+|D(v)|^2+|p|^2 \right)\ dx.
$$
Using the dominated convergence Theorem and \eqref{new020}, we can pass to the limit $R\to \infty$ and we deduce the result.
\end{proof}

Using the energy equation \eqref{drag}, we can rewrite the drag functional \eqref{workdrag} in the following way:
\begin{corollary}
Let $(v,p,\vperp)$ be the solution of \eqref{ncy0.0}--\eqref{ncy0.51} obtained in Theorem \ref{Tmain}.
Then
\begin{equation*}
W(v,p)
=2  \|D(v)\|_{2,\Omega}^2
+\frac{1}{2}\int_{\partial\Omega}(v_*+v_*^{\mathcal C})\cdot n |V+v_*+v_*^{\mathcal C}|^2\,d\gamma.
\end{equation*}
\label{dragbyenergy}
\end{corollary}

From now on, we assume
\begin{equation}\label{icd0.0}
|\xi|\leq c_0, \qquad |\omega|\leq c_0
\end{equation}
and, given $\kappa\in (0,c_0]$, we define
\begin{equation}\label{icd0.1}
\mathcal{V}_\tau^{\kappa} :=\left\{\vstar \in \mathcal{V}_\tau \ ; \
\|\vstar\|_{W^{3/2,2}(\partial\Omega)}\leq \kappa \right\},
\end{equation}
\begin{equation}\label{icd0.2}
\mathcal{V}_{\Gamma}^{\kappa}:=\left\{\vstar \in \mathcal{V}_{\Gamma} \ ; \
\|\vstar\|_{W^{3/2,2}(\partial\Omega)}\leq \kappa\right\},
\end{equation}
where $\mathcal{V}_\tau$ and $\mathcal{V}_{\Gamma}$ are defined by \eqref{defesVtau} and \eqref{defesVchi}, while $c_0$ is the constant in Theorem \ref{Tmain}.

Using Corollary \ref{dragbyenergy}, we deduce that the problem \eqref{infdrag} reduces here to minimize
\begin{equation}\label{icd0.3}
J(v_*)
:=2  \|D(v)\|_{2,\Omega}^2
+\frac{1}{2}\int_{\partial\Omega}(v_*+v_*^{\mathcal C})\cdot n |V+v_*+v_*^{\mathcal C}|^2\,d\gamma,
\end{equation}
where $(v,p,\vperp)$  is the solution of \eqref{ncy0.0}--\eqref{ncy0.51} associated with either $v_*\in \mathcal{V}_\tau^{\kappa}$ or $v_*\in \mathcal{V}_{\Gamma}^{\kappa}$.
This functional $J$ is well-defined since Theorem \ref{Tmain} allows us to define the 
control-to-state mapping $v_* \mapsto (v,p,\vperp)$. In the 
 following sections, we study this mapping and consider the  
optimal control problems:
\begin{equation}\label{Ptaus}
\inf_{v_* \in \mathcal{V}_\tau^{\kappa}} J(v_*)
\end{equation}
or
\begin{equation}\label{Pchis}
\inf_{v_* \in \mathcal{V}_{\Gamma}^{\kappa}} J(v_*).
\end{equation}
Under the condition \eqref{icd0.0}, in the next section, it turns out that \eqref{Ptaus} and \eqref{Pchis} respectively admit solutions for every $\kappa \in (0,c_0]$, however, the radius $\kappa$ of the admissible closed balls \eqref{icd0.1}--\eqref{icd0.2} as well as $(\xi,\omega)$ should be taken still smaller (see Theorems \ref{zhz} and \ref{T1}) in order to characterize the optimal solution in Theorem \ref{main}.

\section{Existence of optimal controls}\label{sec_exi}
Here we show that problems \eqref{Ptaus} and \eqref{Pchis} have a solution, so that the infima are actually minima.
\begin{theorem}\label{Texi}
Assume that $\xi,\omega \in {\mathbb R}^3$ satisfy \eqref{icd0.0}. Let $\kappa\in (0,c_0]$. 
Then each of the optimal control problems \eqref{Ptaus} and \eqref{Pchis} admits a solution.
\end{theorem}
\begin{proof} 
We only consider problem \eqref{Pchis}, the case of tangential controls can be treated with exactly the same arguments. By the embedding $W^{3/2,2}(\partial \Omega)\subset L^3(\partial \Omega)$ and
\eqref{icd0.2}, we have
$$
\|v_*\|_{L^3(\partial \Omega)} \leq C\kappa
$$
and from \eqref{main-est} it follows that
$$
\|v_*^{\mathcal C}\|_{L^3(\partial \Omega)} \leq C(\kappa +c_0).
$$
Therefore,
the functional $J$ is bounded from below on 
$\mathcal{V}_{\Gamma}^{\kappa}$ since
$$
J (v_*) \geq
- C \int_{\partial \Omega} \Big(|v_*+v_*^{\mathcal C}|^3+|v_*+v_*^{\mathcal C}|\Big) \ d\gamma.
$$
Thus, there exists a sequence
$$
\{ v_{*k} \}_{k\in {\mathbb N}}\subset \mathcal{V}_{\Gamma}^{ \kappa}  \text{ such that } 
J(v_{*k}) \to \inf_{v_* \in \mathcal{V}_{\Gamma}^{\kappa}} J(v_*) \quad (k \to \infty).
$$
We will denote by $\{(v_k,  p_k,v_{*k}^{\mathcal C})\}_{k\in {\mathbb N}}$ the corresponding sequence of states, that is,  $(v_k,  p_k,v_{*k}^{\mathcal C})$ is the solution of problem \eqref{ncy0.0}--\eqref{ncy0.51} for $v_{*k}$. Since the admissible set $\mathcal{V}_{\Gamma}^{\kappa}$ is weakly sequentially compact, there exist
$\widehat{v}_{*}\in \mathcal{V}_{\Gamma}^{\kappa}$ and a subsequence of $\{ v_{*k} \}_{k\in {\mathbb N}}$, still denoted by $\{ v_{*k} \}_{k\in {\mathbb N}}$,  such that
\begin{equation}\label{11:58}
{\vstar}_k \rightharpoonup \widehat{v}_*  \;\text{ weakly in }  W^{3/2,2}(\partial \Omega)
\end{equation}
 In what follows we take suitable subsequences in order although they are always denoted by the same symbol $\{v_k\}$ etc.
From \eqref{main-est}--\eqref{main-est-add}, we also deduce the existence of $(\widehat{v}, \widehat{p},\widehat{v}_*^{\mathcal C})$ such that
\begin{gather}
\displaystyle \big(\nabla  v_k,  p_k \big) \rightharpoonup 
	\big(\nabla \widehat{v}, \widehat{p} \big) \;\text{ weakly in } W^{1,2}(\Omega), \label{12:45}\\
v_{*k}^{\mathcal C} \to  \widehat v_*^{\mathcal C}  \;\text{ strongly in }  W^{3/2,2}(\partial \Omega). \label{11:56}
\end{gather}
Concerning \eqref{11:56}, what we see at once is that
${\mathcal C}_\tau$ (resp. ${\mathcal C}_\Gamma$) $\ni v_{*k}^{\mathcal C}$ tends to $\widehat v_*^{\mathcal C}$ weakly in $W^{3/2,2}(\partial\Omega)$ along a subsequence as $k\to\infty$, which yields $\widehat v_*^{\mathcal C}\in {\mathcal C}_\tau$ (resp. ${\mathcal C}_\Gamma$) since the subspace is weakly closed; then, we eventually obtain the strong convergence above because it is a finite dimensional space.
Since $v_k(x)$ tends to zero as $|x|\to\infty$, from a classical embedding inequality we also deduce  (see, for instance, \cite[Theorem II.6.1]{G})
$$
\|v_k\|_{6,\Omega} \leq C
$$
and thus
\begin{equation}
 v_k \rightharpoonup \widehat{v} \;\text{ weakly in } L^{6}(\Omega).
\label{convl6}
\end{equation}

We also have
$$
\varpi  v_k \rightharpoonup \varpi \widehat{v} \;\text{ weakly * in } L^{\infty}(\Omega).
$$
Indeed, there exists $\vartheta \in L^\infty(\Omega)$ such that
$$
\int_{\Omega} \varpi(x)  v_k(x) u(x)dx \to \int_{\Omega} \vartheta (x) u(x)dx, \quad \forall u \in L^1(\Omega).
$$
and, in particular, 
\begin{equation}
\int_{\Omega} \varpi(x)  v_k(x)  \varphi(x)dx \to \int_{\Omega} \vartheta (x)  \varphi(x)dx, \quad \forall \varphi \in C_0^\infty(\overline\Omega).
\label{convl1}
\end{equation}
On the other hand, due to \eqref{convl6}, 
\begin{equation}
\int_{\Omega} \varpi(x)  v_k(x) \varphi(x)dx \to \int_{\Omega} \varpi(x)  v(x) \varphi(x)dx, \quad \forall \varphi \in C_0^\infty(\overline\Omega).
\label{convd}
\end{equation}
From \eqref{convl1} and \eqref{convd} it follows that
$$
\int_{\Omega} \vartheta (x)  \varphi(x)dx = \int_{\Omega} \varpi(x)  v(x) \varphi(x)dx, \quad \forall \varphi \in C_0^\infty(\overline\Omega)
$$
and therefore $\vartheta = \varpi v$ a.e. in $\Omega$.

Finally, using classical compactness results, we also deduce from \eqref{main-est} that
$$
(\nabla v_k, p_k)\to (\nabla \widehat{v},\widehat{p})  \;\text{ strongly in } \ L^2_{loc}(\overline{\Omega}),
$$
$$
\displaystyle v_k \to \widehat{v} \;\text{ strongly in } \ L^2_{loc}(\overline{\Omega}),
$$
and
\begin{equation}\label{12:30}
v_{*k} \to \widehat v_* \;\text{ strongly in } \ L^3(\partial \Omega).
\end{equation}
Using the above convergences, we can pass to the limit in \eqref{weakform} where $v$ is replaced by $v_k$.
We also know that $\widehat v$ satisfies \eqref{ncy0.1} and \eqref{ncy0.3}. We have
$$
\|v_k-\widehat{v}\|_{2,\partial\Omega}
\leq C\|v_k-\widehat{v}\|_{2,\Omega_R}^{1/2}\|v_k-\widehat{v} \|_{1,2,\Omega_R}^{1/2}
\to 0\quad (k\to \infty), 
$$
and using \eqref{11:56} and \eqref{11:58}, we deduce \eqref{ncy0.2}.

Similarly,
\begin{multline*}
\|\sigma(v_k,p_k)n -\sigma(\widehat{v},\widehat{p} )n\|_{2,\partial\Omega} \\
\leq C\|\nabla v_k -\nabla \widehat{v} \|_{2,\Omega_R}^{1/2}\|\nabla v_k - \nabla \widehat{v} \|_{1,2,\Omega_R}^{1/2}
	+C\|p_k - \widehat{p} \|_{2,\Omega_R}^{1/2}\|p_k - \widehat{p} \|_{1,2,\Omega_R}^{1/2}
\to 0 \quad (k \to \infty),
\end{multline*}
so that, using \eqref{12:30} and \eqref{11:56}, we deduce \eqref{ncy0.4} and \eqref{ncy0.5}.

The pair $(\widehat v,\widehat v_*^{\mathcal C})$ is in the class 
 satisfying \eqref{main-est} and thus, by the assumption \eqref{icd0.0} and the uniqueness result in Theorem \ref{Tmain}, $(\widehat v,\widehat p,\widehat v_*^{\mathcal C})$ is the state associated with $\widehat v_*$. 
Now using \eqref{12:45}, \eqref{12:30} and \eqref{11:56} we deduce
$$
J(\widehat v_*) \leq \liminf_{k\to \infty} J(v_{*k})
$$
and therefore $\widehat v_*$ is actually a minimizer:
$$
J(\widehat v_*)=  \inf_{v_* \in \mathcal{V}_{\Gamma}^{\kappa}} J(v_*).
$$
\end{proof}

\section{Regularity of the control-to-state mapping}\label{sec_reg}
Let 
$$
\mathcal{W}:=\left\{ (v,v_*^{\mathcal C}) \ ;\
v \in D^{1,2}(\Omega)\cap L^\infty_{1,\varpi}(\Omega),\,\nabla \cdot v = 0 \mbox{ in }\Omega,\,
 v_*^{\mathcal C} \in W^{3/2,2}(\partial\Omega) \right\}
$$
endowed with the norm
$$
\| (v,v_*^{\mathcal C}) \|_{\mathcal{W}} 
:=|v|_{1,2,\Omega} +  \lceil v \rceil_{1,\varpi,\Omega} 
+\|v_*^{\mathcal C}\|_{3/2,2,\partial\Omega}
= \|\nabla v\|_{2,\Omega} + \lceil v\rceil_{1,\varpi,\Omega}
 +\|v_*^{\mathcal C}\|_{3/2,2,\partial\Omega},
$$
and consider the subset
\begin{equation}
\mathcal{W}_{R_0}:=\left\{
(v,v_*^{\mathcal C})
\in \mathcal{W} \ ; \ \| (v,v_*^{\mathcal C}) \|_{\mathcal{W}} 
\leq R_0\right\},
\label{WR0}
\end{equation}
where $R_0$ is the same as in \eqref{inv-ball}, that is the right-hand side of \eqref{main-est}.
 Suppose \eqref{icd0.0}. Given $\kappa\in (0,c_0]$,
Theorem \ref{Tmain} allows us to define the following control-to-state
mappings (recall \eqref{icd0.1} and \eqref{icd0.2})
\begin{equation}\label{lis1.2}
\Lambda_\tau : \mathcal{V}_\tau^{\kappa} \to 
\mathcal{W}_{R_0}\cap\big[(D^{2,2}(\Omega)\cap L^2(\Omega))\times \mathcal{C}_\tau\big],
\quad
v_* \mapsto (v,v_*^{\mathcal C})
\end{equation}
\begin{equation}\label{lis1.2chi}
\Lambda_\chi : \mathcal{V}_{\Gamma}^{\kappa} \to 
 \mathcal{W}_{R_0}\cap\big[(D^{2,2}(\Omega)\cap L^2(\Omega))\times \mathcal{C}_\chi\big],
\quad
v_* \mapsto (v,v_*^{\mathcal C})
\end{equation}
where 
 $(v,\vperp)$ together with the pressure $p\in W^{1,2}(\Omega)$
is the solution of problem \eqref{ncy0.0}--\eqref{ncy0.51} associated with $v_*$ given in the admissible set 
$\mathcal{V}_\tau^{\kappa}$ or
$\mathcal{V}_{\Gamma}^{\kappa}$.

Our aim is to show that the maps $\Lambda_\tau$ and $\Lambda_\chi$ are G\^ateaux differentiable. We analyze this problem in detail for $\Lambda_\tau$, the idea being similar for the mapping $\Lambda_\chi$. 

In order to compute the G\^ateaux derivative of $\Lambda_\tau$ at $v_*\in \mathcal{V}_\tau^{\kappa}$ in the direction $\overline{v_*} \in {\mathcal V}_\tau$, that is denoted by
$$
D \Lambda_\tau (v_*) {\overline{v_*}} = (z,z_*^{\mathcal C}),
$$
we suppose that $v_*+h \overline{v_*}$ is also in $\mathcal{V}_\tau^{\kappa}$ for $0 < h < h_0$
 with sufficiently small $h_0$; in fact, this is accomplished as long as 
$\langle v_*,\overline{v_*}\rangle_{W^{3/2,2}(\partial\Omega)}<0$
even if $\|v_*\|_{3/2,2,\partial\Omega}=\kappa$.

We consider 
$$
(v_h, v_{h*}^{\mathcal C},p_h)\in 
\Big(\mathcal{W}_{R_0}\cap\big[(D^{2,2}(\Omega)\cap L^2(\Omega))\times \mathcal{C}_\tau\big]\Big)
\times W^{1,2}(\Omega)
$$
which is the solution to \eqref{ncy0.0}--\eqref{ncy0.51} associated with $v_*+h \overline{v_*}$,
 that is, $(v_h, v_{h*}^{\mathcal C}):=\Lambda_\tau(v_*+h\overline{v_*})$.
Then 
$$ 
\begin{array}{rcl}
 (z_h,z_{h*}^{\mathcal C},r_h)  &:= &\displaystyle \frac{(v_h,v_{h*}^{\mathcal C},p_h) - (v,v_*^{\mathcal C},p)}{h} \medskip \\
& = & \displaystyle \frac{(\Lambda_\tau (v_*+h \overline{v_*}),p_h) - (\Lambda_\tau (v_*),p)}{h}\in 
\Big(\mathcal{W}\cap\big[(D^{2,2}(\Omega)\cap L^2(\Omega))\times \mathcal{C}_\tau\big]\Big)
\times W^{1,2}(\Omega)
\end{array}
$$ 
satisfies the following system
\begin{gather}
-\div \sigma(z_h,r_h) + v \cdot \nabla z_h + z_h \cdot \nabla v_h  -  V \cdot \nabla z_h + \omega\times z_h =
0 \quad \text{in} \ \Omega \label{zhum}\\
\div z_h =0 \quad \text{in} \ \Omega \label{zh2} \\
z_h =   z_{h*}^{\mathcal C} + \overline{v_*} \quad \text{on}\ \partial \Omega \label{zh3}\\
\lim_{|x|\to \infty} z_h(x) = 0  \label{zh4}\\
\int_{\partial \Omega} \sigma(z_h,r_h)n \ d \gamma = 0 \label{zh5}
\\
\int_{\partial \Omega} x\times \sigma(z_h,r_h)n \ d \gamma 
 =  0 , \label{zh6}
 \\
z_{h*}^{\mathcal C}\in \mathcal{C}_\tau. \label{zulth}
\end{gather}
Our aim is to show that, when $h \to 0$, $(z_h,z_{h*}^{\mathcal C},r_h)$ converges to $(z,z_{*}^{\mathcal C},r)$, which solves the linearized state equations 
\begin{gather}
-\div \sigma(z,r)  + v \cdot \nabla z + z \cdot \nabla v -  V \cdot \nabla z + \omega\times z = 0 \quad \text{in} \ \Omega \label{11:17}\\
\div z =0 \quad \text{in} \ \Omega \\
z =   z_*^{\mathcal C} + \overline{v_*} \quad \text{on}\ \partial \Omega \\
\lim_{|x|\to \infty} z(x) = 0  \label{15:15}\\
\int_{\partial \Omega} \sigma(z,r)n \ d \gamma =0 \label{17:38} \\
 \int_{\partial \Omega} x\times \sigma(z,r)n \ d \gamma =0
\\
z_{*}^{\mathcal C}\in \mathcal{C}_\tau. \label{11:18}
\end{gather}
In the case of localized controls, instead of the equations \eqref{zh5}--\eqref{zulth}, we have
\begin{gather}
\int_{\partial\Omega} 
[- \sigma(z_h,r_h) + (v_* + v_*^{\mathcal C})\cdot n  ( \overline{v_*} + z_{h*}^{\mathcal C} ) + ( \overline{v_*} + z_{h*}^{\mathcal C} )\cdot n (v_*+h\overline{v_*} + v_{h*}^{\mathcal C} + V + \omega\times x)]\,d\gamma = 0
\label{zh5loc}
\\
\int_{\partial\Omega} x \times 
[- \sigma(z_h,r_h) + (v_* + v_*^{\mathcal C})\cdot n  ( \overline{v_*} + z_{h*}^{\mathcal C} ) + ( \overline{v_*} + z_{h*}^{\mathcal C} )\cdot n (v_*+h\overline{v_*} + v_{h*}^{\mathcal C} + V + \omega\times x)]\,d\gamma = 0 \label{zh6loc} \\
 z_{h*}^{\mathcal C}\in \mathcal{C}_{\chi} \label{zulthloc}
 \end{gather}
 and in the corresponding linearized state equations conditions \eqref{17:38}--\eqref{11:18} are replaced by
\begin{gather}
\int_{\partial \Omega} \left[-\sigma(z,r)n + (v_*+v_*^{\mathcal C})\cdot n  (\overline v_* +z_*^{\mathcal C}) +  (\overline{v_*}+z_{*}^{\mathcal C})\cdot n (v_* +v_*^{\mathcal C}+ V + \omega \times x)
  \right] \ d \gamma =0
\label{ncy0.4-old}\\
\int_{\partial \Omega} x\times \left[-\sigma(z,r)n + (v_*+v_*^{\mathcal C})\cdot n  (\overline v_* +z_*^{\mathcal C}) +  (\overline{v_*}+z_{*}^{\mathcal C})\cdot n (v_* +v_*^{\mathcal C}+ V + \omega \times x)
  \right] \ d \gamma =0
\label{ncy0.5-old}
\\
z_{*}^{\mathcal C}\in {\mathcal C}_\chi.
\label{ncy0.51-old}
\end{gather}
More specifically, we have
\begin{theorem}
There exists a constant $\kappa_1\in (0,c_0]$ 
depending on $\Omega$
such that if
$\xi, \omega\in \mathbb R^3$ and $v_*\in {\mathcal V}_\tau$ (resp. $v_*\in {\mathcal V}_\Gamma$) satisfy
\begin{equation}
|\xi|\leq\kappa_1, \qquad |\omega|\leq\kappa_1, \qquad
\|v_*\|_{3/2,2,\partial\Omega}\leq\kappa_1,
\label{small-1}
\end{equation}
then the following assertion holds, where $c_0$ is the constant in Theorem \ref{Tmain}:

Let $(v, v_*^{\mathcal C},p)$ be the solution to the state equations \eqref{ncy0.0}--\eqref{ncy0.51} associated with $v_*\in \mathcal{V}_\tau^{\kappa_1}$ (resp. $v_*\in \mathcal{V}_\Gamma^{\kappa_1}$) obtained in Theorem \ref{Tmain}. Then the mapping $\Lambda_\tau$ (resp. $\Lambda_\chi$)
defined by \eqref{lis1.2} (resp. \eqref{lis1.2chi}) with $\kappa=\kappa_1$
is G\^ateaux differentiable with values in
$\big[W^{2,2}(\Omega)\times L^\infty_{1,\varpi}(\Omega)\big]\times W^{3/2,2}(\partial\Omega)$
at $v_*\in \mathcal{V}_\tau^{\kappa_1}$ (resp. $v_*\in \mathcal{V}_\Gamma^{\kappa_1}$) in the direction $\overline{v_*}\in {\mathcal V}_\tau$ (resp. $\overline{v_*}\in {\mathcal V}_\Gamma$), where $\overline{v_*}$ must be taken such that $\langle v_*,\overline{v_*}\rangle_{W^{3/2,2}(\partial\Omega)} <0$ if $\|v_*\|_{3/2,2,\partial\Omega}=\kappa_1$, whereas it can be arbitrary if $\|v_*\|_{3/2,2,\partial\Omega}<\kappa_1$, and its derivative is given by $D \Lambda_\tau (v_*) {\overline{v_*}} = (z,z_*^{\mathcal C})$, 
(resp. $D \Lambda_\chi (v_*) {\overline{v_*}} = (z,z_*^{\mathcal C})$)
with
$(z,z_*^{\mathcal C},r)$ being the solution to the problem \eqref{11:17}--\eqref{11:18}
(resp. \eqref{11:17}--\eqref{15:15}, \eqref{ncy0.4-old}--\eqref{ncy0.51-old}).
Namely, we have 
$$\| z_h - z \|_{2,2,\Omega} + \lceil z_h - z \rceil_{1,\varpi,\Omega} 
 +\|z_{h*}^{\mathcal C}-z_*^{\mathcal C}\|_{3/2,2,\partial\Omega}
\to 0,$$
where $z_h$ satisfies \eqref{zhum}--\eqref{zulth} (resp. \eqref{zhum}--\eqref{zh4}, \eqref{zh5loc}--\eqref{zulthloc}). 
\label{zhz}
\end{theorem}
\begin{proof}
We write a detailed proof for $\Lambda_\tau$ but point out the main differences for the mapping $\Lambda_\chi$.
Equation \eqref{zhum} can be written in the form 
$$
-\div \sigma(z_h,r_h)  -  V \cdot \nabla z_h + \omega\times z_h = f_h
$$
 with 
$$
 f_h : = - (v \cdot \nabla z_h + z_h \cdot \nabla v_h) = - \div\left(z_h\otimes v + v_h \otimes z_h\right) = \div F_h,
$$
$$
F_h := - (z_h\otimes v + v_h \otimes z_h).
$$
If $v,v_h \in {\mathcal W}_{R_0}$, then the following estimates hold for $f_h$ and $F_h$:
$$
\| f_h \|_{2,\Omega}
\leq C(\lceil v\rceil_{1,\varpi,\Omega}  \|  \nabla z_h \|_{2,\Omega} + \lceil z_h \rceil_{1,\varpi,\Omega} \| \nabla v_h \|_{2,\Omega})
\leq C R_0 (|z_h|_{1,2,\Omega} + \lceil z_h \rceil_{1,\varpi,\Omega}),
$$
$$
\lceil F_h \rceil_{2,\varpi,\Omega} = \lceil z_h\otimes v + v_h \otimes z_h \rceil_{2,\varpi,\Omega}
\leq C R_0   \lceil z_h \rceil_{1,\varpi,\Omega},
$$
where $R_0$ is given by \eqref{inv-ball}, see \eqref{WR0}.   
In order to apply \cite[Proposition 4.5]{HST}, 
which is still valid even though $V$ is replaced by $\overline{v_*}$,
we rewrite the conditions \eqref{zh5} and \eqref{zh6} as
\begin{gather}
\int_{\partial \Omega} \left[\sigma(z_h,r_h)n +(V\cdot n)z_h\right] \ d \gamma = \int_{\partial \Omega} (V\cdot n)z_h \ d \gamma \label{zhultnew1}
\\
\int_{\partial \Omega} x\times \left[\sigma(z_h,r_h)n + (V\cdot n)z_h \right] \ d \gamma 
 =  \int_{\partial \Omega} x\times \left[(V\cdot n)z_h \right] \ d \gamma,
  \label{zhultnew2}
\end{gather}
to find  
\begin{multline*}
|z_h|_{1,2,\Omega} +  |z_h|_{2,2,\Omega} + \lceil z_h \rceil_{1,\varpi,\Omega}
+\|r_h\|_{1,2,\Omega}
+\|z_{h*}^{\mathcal C}\|_{3/2,2,\partial\Omega}
\\
\leq C \left[R_0(|z_h|_{1,2,\Omega} + \lceil z_h \rceil_{1,\varpi,\Omega}) + \|\overline{v_*}\|_{3/2,2,\partial \Omega}
+(|\xi|+|\omega| )(\|\overline{v_*}\|_{3/2,2,\partial \Omega}+\|z_{h*}^{\mathcal C}\|_{3/2,2,\partial\Omega})\right].
\end{multline*}
For $R_0$ and $(\xi,\omega)$ small enough, we deduce that $(z_h,r_h,z_{h*}^{\mathcal C})$ is uniformly bounded 
in 
\[
\left[D^{1,2}(\Omega) \cap D^{2,2}(\Omega)\cap L^\infty_{1,\varpi}(\Omega)\right] \times W^{1,2}(\Omega)\times W^{3/2,2}(\partial\Omega).
\]

Now using \eqref{L2-est2} or \eqref{L2-est3}, depending on $\omega=0$ or $\omega\not=0$, we get
$$
\|z_h\|_{2,\Omega}
\leq C
\big[
(1+|\xi|)
\big(\|\nabla z_h \|_{2,\Omega}+|\Phi_h|\big)+\|r_h\|_{2,\Omega}+\|F_h\|_{2,\Omega}
\big]
+C'\||x|F_h\|_{2,\Omega}
$$
or
$$
\begin{array}{rcl}
\|z_h\|_{2,\Omega} & \leq  &\displaystyle  C\left(
|\omega|^{-1/4}+\frac{|\omega\cdot\xi|^{1/2}}{|\omega|}
\right)
\big(|N_h|+|\omega||\Phi_h|\big)
+C'\||x|F_h\|_{2,\Omega}  \\
&&\displaystyle +C\left(
1+\frac{|\omega\times\xi|}{|\omega|^2}
\right)
\left[
(1+|\xi|+|\omega|)
\big(\|\nabla z_h\|_{2,\Omega}+|\Phi_h|\big)+\|r_h\|_{2,\Omega}+\|F_h\|_{2,\Omega}
\right]\end{array}
$$
where
\begin{equation}
\begin{split}
N_h
&:=\int_{\partial\Omega}
[\sigma(z_h,r_h)+z_h\otimes V-(\omega\times x)\otimes z_h+F_h]\,n\,d\gamma \medskip \\
&=\int_{\partial\Omega} 
[\sigma(z_h,r_h) n + z_h (V-v)\cdot n -(v_h + \omega\times x) z_h\cdot n]\,d\gamma \medskip \\
&=\int_{\partial\Omega} 
[\sigma(z_h,r_h)n - (v_* + v_*^{\mathcal C})\cdot n  ( \overline{v_*} + z_{h*}^{\mathcal C} ) - ( \overline{v_*} + z_{h*}^{\mathcal C} )\cdot n (v_*+h\overline{v_*} + v_{h*}^{\mathcal C} + V + \omega\times x)]\,d\gamma =0 
\end{split}
\label{force-zh}
\end{equation}
on account of the conditions \eqref{zh5} for tangential controls (then it is $(v_* + v_*^{\mathcal C})\cdot n=( \overline{v_*} + z_{h*}^{\mathcal C} )\cdot n=0$) or \eqref{zh5loc} if we are considering localized controls; in fact, since $v$ and $v_h$ satisfy \eqref{903}, we should have $N_h=0$ for $z_h$, no matter which kind of controls we would adopt. Moreover,
$$
 \Phi_h  :=  \int_{\partial\Omega}n\cdot z_h\,d\gamma = \int_{\partial\Omega}n\cdot (z_{h*}^{\mathcal C} + \overline{v_*})\,d\gamma 
$$
which is zero if $z_{h*}^{\mathcal C} \in {\mathcal C}_\tau$, $\overline{v_*} \in {\mathcal V}_\tau$. For localized controls, we use the estimate
$$|\Phi_h  | = \left| \int_{\partial\Omega}n\cdot (z_{h*}^{\mathcal C} + \overline{v_*})\,d\gamma \right| \leq C ( \|z_{h*}^{\mathcal C}\|_{3/2,2,\partial\Omega} + \|\overline{v_*}\|_{3/2,2,\partial \Omega} ).$$
Now we can use the estimate
$$
\| (1+|x|)F_h\|_{2,\Omega} = \| (1+|x|)(z_h\otimes v +v\otimes z_h+hz_h\otimes z_h \|_{2,\Omega}
\leq C \lceil z_h \rceil_{1,\varpi,\Omega} (\|v\|_{2,\Omega}+ h \|z_h\|_{2,\Omega})
$$
together with the uniform boundedness of $z_h$, $r_h$ and $z_{h*}^{\mathcal C}$
in $D^{1,2}(\Omega) \cap D^{2,2}(\Omega)\cap L^\infty_{1,\varpi}(\Omega)$,
$W^{1,2}(\Omega)$ and $W^{3/2,2}(\partial\Omega)$, respectively, to conclude that $z_h$ is uniformly bounded in $L^2(\Omega)$ when $h$ is close to zero.

Therefore, there exists
\[
(z,r,z_{*}^{\mathcal C})
\in \big[W^{2,2}(\Omega)\cap L^\infty_{1,\varpi}(\Omega)\big]\times W^{1,2}(\Omega)\times {\mathcal C}_\tau
\]
such that
\begin{gather}
\displaystyle \varpi  z_h \rightharpoonup \varpi z \quad \text{weakly * in} \ L^{\infty}(\Omega),\\
\displaystyle \big( z_h,  r_h \big) \rightharpoonup 
	\big( z, r \big) \quad \text{weakly in} \ W^{2,2}(\Omega) \times W^{1,2}(\Omega), \\
	z_{h*}^{\mathcal C} \to z_{*}^{\mathcal C}  \quad \text{ strongly in } W^{3/2,2}(\partial\Omega),
	\label{10:57}
\end{gather}
along a subsequence as $h \to 0$, where the strong convergence \eqref{10:57} follows from the same reasoning as in \eqref{11:56}.

By classical compactness results, we also have
$$
\displaystyle z_h \to z \quad \text{strongly in} \ W^{1,2}_{loc}(\overline{\Omega}),
$$
and since $v_h=v+h z_h$,
$$
\displaystyle v_h \to v \quad \text{strongly in} \ W^{1,2}_{loc}(\overline{\Omega}).
$$
Proceeding as in the proof in Theorem \ref{Texi}, we can pass to the limit $h\to 0$ in \eqref{zhum}-\eqref{zulth} and show that $(z,r,z_*^{\mathcal C})$ satisfies 
\eqref{11:17}--\eqref{11:18}.

Now we prove the convergence of $z_h-z$ in the norm $\|\cdot\|_{2,2,\Omega} +  \lceil \cdot \rceil_{1,\varpi,\Omega} $. Consider the problem
\begin{gather*}
-\div \sigma(z_h-z,r_h-r) -  V \cdot \nabla (z_h-z) + \omega\times (z_h-z) 
\\
\hspace{5cm}=
-\left[(v \cdot \nabla z_h + z_h \cdot \nabla v_h) - (v \cdot \nabla z + z \cdot \nabla v)  \right] 
=:g_h\quad \text{in} \ \Omega \label{zh1}\\
\div (z_h -z) =0 \quad \text{in} \ \Omega \\
z_h - z =   z_{h*}^{\mathcal C} - z_{*}^{\mathcal C} \quad \text{on}\ \partial \Omega \\
\lim_{|x|\to \infty}  (z_h(x) -z(x)) = 0\\
\int_{\partial \Omega} \sigma(z_h-z,r_h-r)n d \gamma = 0
\\
\int_{\partial \Omega} x\times \sigma(z_h-z,r_h-r)n \ d \gamma 
 =  0,
 \\
z_{h*}^{\mathcal C}-z_*^{\mathcal C}\in \mathcal{C}_\tau. \label{zhult}
\end{gather*}
and notice that, since $v_h-v=h z_h$, we can write
$$
\begin{array}{rcl}
- g_h&:=& (v \cdot \nabla z_h + z_h \cdot \nabla v_h) - (v \cdot \nabla z + z \cdot \nabla v) \medskip \\
&=&  v \cdot \nabla (z_h-z) + (z_h-z)\cdot \nabla v_h + z \cdot \nabla (v_h-v)   \medskip \\
&=& v \cdot \nabla (z_h-z) + (z_h-z)\cdot \nabla v + hz_h\cdot \nabla z_h  
\end{array}
$$
and
$$
- g_h =  \div\left((z_h-z)\otimes v + v \otimes (z_h-z) + h z_h \otimes  z_h  \right) = : - \div G_h.
$$
We use the fact that $z_h$ is uniformly bounded 
in $W^{2,2}(\Omega)\cap L^\infty_{1,\varpi}(\Omega)$ to get the following estimates
$$
\begin{array}{rcl}
\| g_h \|_{2,\Omega} & = & \| v \cdot \nabla (z_h-z) + (z_h-z)\cdot \nabla v + hz_h\cdot \nabla z_h    \|_{2,\Omega} \medskip \\
& \leq & C \left[R_0 (|z_h-z|_{1,2,\Omega}  + \lceil z_h -z\rceil_{1,\varpi,\Omega})  +  h \lceil z_h \rceil_{1,\varpi,\Omega} |z_h|_{1,2,\Omega}\right],
\end{array}
$$
$$
\lceil G_h \rceil_{2,\varpi,\Omega} = \lceil (z_h-z)\otimes v + v \otimes (z_h-z) + h z_h \otimes z_h \rceil_{2,\varpi,\Omega}  \leq  C \left( R_0  \lceil z_h - z \rceil_{1,\varpi,\Omega} +  h \lceil z_h \rceil_{1,\varpi,\Omega}^2 \right),
$$
with $R_0$ as specified above. Then using \cite[Proposition 4.5]{HST}, we first deduce that  
\begin{equation*}
\begin{split}
&\quad 
|z_h-z|_{1,2,\Omega}  + |z_h-z|_{2,2,\Omega} +  \lceil z_h-z \rceil_{1,\varpi,\Omega}
+\|r_h-r\|_{1,2,\Omega}
+\|z_{h*}^{\mathcal C}-z_{*}^{\mathcal C}\|_{3/2,2,\partial\Omega} \\
&\leq C \left[R_0(\lceil z_h -z\rceil_{1,\varpi,\Omega} + |z_h-z|_{1,2,\Omega} ) 
 +(|\xi|+|\omega| )\|z_{h*}^{\mathcal C}-z_{*}^{\mathcal C}\|_{3/2,2,\partial\Omega}
+ h \lceil z_h \rceil_{1,\varpi,\Omega} |z_h|_{1,2,\Omega} +  h \lceil z_h \rceil_{1,\varpi,\Omega}^2 
\right].
\end{split}
\end{equation*}
This estimate yields a first convergence result for $z_h-z$:
\begin{equation}
|z_h-z|_{1,2,\Omega}  + |z_h-z|_{2,2,\Omega} +  \lceil z_h-z \rceil_{1,\varpi,\Omega}
+\|r_h-r\|_{1,2,\Omega} 
+\|z_{h*}^{\mathcal C}-z_{*}^{\mathcal C}\|_{3/2,2,\partial\Omega} \to 0 \quad (h\to 0)
\label{str-conv}
\end{equation}
under suitable smallness assumptions on $R_0$ as well as $(\xi,\omega)$
and the uniform boundedness of $z_h$ already established.

Let
\begin{equation}
N:=\int_{\partial\Omega}[\sigma(z,r)+z\otimes V-(\omega\times x)\otimes z+F]n\,d\gamma,
\label{force-z}
\end{equation}
where $F=-(z\otimes v+v\otimes z)$.
Then we have
\[
|N_h-N|\leq C\left\|\big(\nabla (z_h-z), r_h-r\big)\right\|_{2,\partial\Omega}
+C(1+\|v\|_{2,\partial\Omega})\|z_{h*}^{\mathcal C}-z_*^{\mathcal C}\|_{2,\partial\Omega}+Ch\|z_{h*}^{\mathcal C}+\overline{v_*}\|_{2,\partial\Omega}^2
\]
which goes to zero as $h\to 0$ by \eqref{str-conv}, yielding $N=0$ because of $N_h=0$, see \eqref{force-zh}.

By \eqref{L2-est3}, in the case $\omega \not=0$ (the case $\omega=0$ is even simpler), we have
\begin{multline*}
\|z_h - z\|_{2,\Omega}  \leq C\left(
|\omega|^{-1/4}+\frac{|\omega\cdot\xi|^{1/2}}{|\omega|}
\right)
\big(|M_h|+|\omega||\Psi_h|\big)
+C'\||x|G_h\|_{2,\Omega}  \\
+C\left(
1+\frac{|\omega\times\xi|}{|\omega|^2}
\right)
\left[
(1+|\xi|+|\omega|)
\big(\|\nabla (z_h-z)\|_{2,\Omega}+|\Psi_h|\big)+\|r_h-r\|_{2,\Omega}+\|G_h\|_{2,\Omega}
\right],
\end{multline*}
where  
\begin{equation}
M_h := \displaystyle \int_{\partial\Omega}
[\sigma(z_h-z,r_h-r)+(z_h-z)\otimes V-(\omega\times x)\otimes (z_h-z)+G_h]\,n\,d\gamma \medskip
=N_h-N=0
\label{force-diff}
\end{equation}
on account of $N_h=N=0$, see \eqref{force-zh} and \eqref{force-z}, no matter which kind of controls we would adopt, while we have
$$
\Psi_h : = \int_{\partial\Omega}n\cdot (z_h-z)\,d\gamma \to 0
$$
as $h\to 0$ since
$|\Psi_h|\leq C\|z_{h*}^{\mathcal C}-z_*^{\mathcal C}\|_{2,\partial\Omega}$.
Hence, the estimate
$$
\|(1+|x|) G_h \|_{2,\Omega} \leq  C \left( \|v\|_{2,\Omega} \lceil z_h - z \rceil_{1,\varpi,\Omega} +  h \lceil z_h \rceil_{1,\varpi,\Omega}\| z_h \|_2 \right)
$$
and the previous convergence results \eqref{str-conv}  
yield $\|z_h - z\|_{2,\Omega} \to 0$ when $h\to 0$.
We have completed the proof provided that $R_0$ given by \eqref{inv-ball} is small enough as we have mentioned twice, which is accomplished through \eqref{small-1} with some $\kappa_1\in (0,c_0]$.
\end{proof}

\section{Necessary first order conditions for an optimal control}\label{sec_main}
In this section, we introduce the Lagrangian associated with problems \eqref{Ptaus} and \eqref{Pchis}, analyze the adjoint system and obtain a characterization of the optimal controls.
\subsection{Introduction of the Lagrangian}
Let us define
\begin{equation}\label{lis0.2}
\mathcal{Y} := \left\{v\in  W^{2,2}(\Omega) \cap L^\infty_{1,\varpi}(\Omega)   \ ;\  \nabla \cdot v = 0 \mbox{ in } \Omega\right\},
\end{equation}
\begin{multline}\label{defcalu}
\mathcal{U} := \left\{ u \in 
 L^6(\Omega)\cap D^{1,2}(\Omega)\cap D^{2,2}(\Omega) \ ;\  
V\cdot \nabla u-\omega\times u\in L^2(\Omega),
\
\nabla \cdot u = 0 \mbox{ in } \Omega, \right.
\\
\left. \exists \ell_u,k_u\in \mathbb{R}^3\ 
u=\ell_u + k_u \times x \quad \text{on} \ \partial \Omega\right\},
\end{multline}
\begin{equation}
\mathcal{Z} := L^2(\partial \Omega).
\label{test-bdy}
\end{equation}
Using these spaces, we can obtain a weak formulation for our problems
 \eqref{ncy0.0}--\eqref{ncy0.51}.
\begin{proposition}\label{P1}
Assume 
\begin{equation}\label{opt0.0}
(v,p,\vperp)\in \mathcal{Y}\times W^{1,2}(\Omega)\times  \mathcal{C}_\chi
\end{equation}
is a solution of \eqref{ncy0.0}--\eqref{ncy0.51} associated with $v_*\in \mathcal{V}_{\Gamma}$. Then
\begin{multline}\label{ncy0.6}
2 \int_{\Omega}  Dv: Du \ dx
- \int_{\Omega} (v\cdot \nabla u )\cdot v\ dx
+\int_{\Omega} \left( V\cdot \nabla u- \omega\times u \right) \cdot v \ dx
\\
-m (\xi\times \omega)\cdot \ell_u 
-((I \omega)\times \omega)\cdot k_u
-
\int_{\partial \Omega} \left((v_*+v_*^{\mathcal C})\cdot n \right) (\omega \times x)\cdot (\ell_u+k_u\times x) \ d \gamma
=0, \quad \forall u\in \mathcal{U},
\end{multline}
\begin{equation}\label{lis0.0}
 \langle v-V- \vstar-\vperp,\zeta \rangle_{\partial \Omega}=0,
 \quad \forall \zeta\in \mathcal{Z}.
\end{equation}
If $(v,p,\vperp)\in \mathcal{Y}\times W^{1,2}(\Omega)\times  \mathcal{C}_\tau$ is a solution of \eqref{ncy0.0}--\eqref{ncy0.51} associated with $v_*\in \mathcal{V}_\tau$ then, a similar result holds true, but instead of \eqref{ncy0.6}, we have
\begin{multline}\label{ncy0.6tau}
2 \int_{\Omega}  Dv: Du \ dx
- \int_{\Omega} (v\cdot \nabla u) \cdot v\ dx
+\int_{\Omega} \left( V\cdot \nabla u- \omega\times u \right) \cdot v \ dx
\\
-m (\xi\times \omega)\cdot \ell_u 
-((I \omega)\times \omega)\cdot k_u
=0, \quad \forall u\in \mathcal{U}.
\end{multline}

Conversely, if $(v,\vperp)\in \mathcal{Y}\times \mathcal{C}_\chi$
(resp. $\mathcal{Y}\times \mathcal{C}_\tau$)
satisfies \eqref{ncy0.6} (resp. \eqref{ncy0.6tau}) together with \eqref{lis0.0}, then there exists
$p\in W^{1,2}(\Omega)$ such that \eqref{ncy0.0}--\eqref{ncy0.51} hold.

\end{proposition}
\begin{proof}
We use the same cut-off function $\psi_R$ as in the final stage of the proof of Theorem \ref{Tmain}.
Recalling that $(\omega \times x )\cdot \nabla \psi_R=0$ and \eqref{cut-est}, from which we obtain
\[
\|V\cdot \nabla \psi_R\|_{3,\mathbb R^3} 
= \|\xi \cdot \nabla \psi_R\|_{3,\mathbb R^3}=C|\xi|
\]
with a constant $C>0$ independent of $R$.
Assume $u\in \mathcal{U}$, multiply \eqref{ncy0.0} by $\psi_R u$ and integrate by parts:
\begin{equation}
\label{opt0.4}
\begin{array}{rcl}
0 & =  & 
\displaystyle -\int_{\Omega}\div \sigma(v,p)\cdot (\psi_R u) \ dx
+ \int_{\Omega} ((v-V)\cdot \nabla v)\cdot (\psi_R u)\ dx
+\int_{\Omega}( \omega\times v)\cdot (\psi_R u) \ dx
 \medskip \\
  &  =  & 
\displaystyle -\int_{\partial \Omega} \sigma(v,p)n \cdot u \ d\gamma
+2 \int_{\Omega}  \psi_R D(v): D(u) \ dx
- \int_{\Omega} \psi_R \big((v-V)\cdot \nabla u\big) \cdot v\ dx
\medskip \\
&  & \displaystyle + \int_{\partial \Omega} (v-V)\cdot n  (v\cdot u) \ dx
-\int_{\Omega}\psi_R (\omega\times u)\cdot v \ dx
\medskip \\
&  & \displaystyle + \int_{\Omega}   (2 D(v)-p \mathbb{I}_3): (u \otimes \nabla \psi_R) \ dx
- \int_{\Omega} \nabla \psi_R \cdot (v-V) (u \cdot v)\ dx.
\end{array}
\end{equation}
On the other hand, using \eqref{ncy0.4} and \eqref{ncy0.5}, we have
$$
m (\xi\times \omega)\cdot\ell_u 
+((I \omega)\times \omega)\cdot k_u
+
\int_{\partial \Omega} \left[-\sigma(v,p)n +  \left((v_*+v_*^{\mathcal C})\cdot n \right) (v_* +v_*^{\mathcal C}+ V + \omega \times x)
  \right]\cdot u \ d \gamma
 =0.
$$
Combining the above relation with \eqref{opt0.4} yields 
\begin{multline}\label{18:08}
0=2 \int_{\Omega}  \psi_R D(v): D(u) \ dx
- \int_{\Omega} \psi_R  v\cdot \nabla u \cdot v\ dx
+\int_{\Omega}\psi_R ( V\cdot \nabla u- \omega\times u)\cdot v \ dx
\\
+ \int_{\Omega}   (2 D(v)-p \mathbb{I}_3): (u \otimes \nabla \psi_R) \ dx
- \int_{\Omega} \nabla \psi_R \cdot (v-V) (u \cdot v)\ dx
\\
-m (\xi\times \omega)\cdot\ell_u 
-((I \omega)\times \omega)\cdot k_u
-
\int_{\partial \Omega}  \left((v_*+v_*^{\mathcal C})\cdot n \right) (\omega \times x) \cdot u \ d \gamma.
\end{multline}
Recalling \eqref{cut-est} together with the summability properties given in \eqref{lis0.2} and \eqref{defcalu}, we get
\begin{equation}\label{18:11}
\left| \int_{\Omega}   (2 D(v)-p \mathbb{I}_3): (u \otimes \nabla \psi_R) \ dx\right| \leq \| \nabla \psi_R \|_{3,A_{R,2R}}\left(\|D(v)\|_{2,\Omega}+\|p\|_{2,\Omega}\right) \|u\|_{6,A_{R,2R}} \to 0 \quad \text{as} \ R\to \infty
\end{equation}
and
\begin{equation}\label{17:53}
\left| \int_{\Omega} \nabla \psi_R \cdot (v-V) (u \cdot v)\ dx\right| 
\leq C \left(\frac 1R \|v\|^2_{2,\Omega} \|u\|_{\infty,\Omega}
+ \| \nabla \psi_R \|_{3,A_{R,2R}}\|v\|_{2,\Omega} \|u\|_{6,A_{R,2R}} \right).
\end{equation}
Hence, letting $R\to \infty$ in \eqref{18:08}, yields \eqref{ncy0.6}.

Conversely, by taking $u\in \mathcal{D}(\Omega)$ in \eqref{ncy0.6}, we find that there exists 
 $p\in L^2_{\rm loc}(\overline{\Omega})$ such that \eqref{ncy0.0} holds. 
Then applying \cite[Proposition 2.1]{HST} with $v\otimes v\in L^\infty_{2,\varpi}(\Omega)$ and $f=-\div(v\otimes v)\in L^2(\Omega)$, we deduce that
$p\in W^{1,2}(\Omega)$. Finally, multiplying \eqref{ncy0.0} by $\psi_R u\in \mathcal{U}$ as above, taking $R\to \infty$ and comparing with 
\eqref{ncy0.6}, we obtain \eqref{ncy0.4} and \eqref{ncy0.5}.
\end{proof}

\begin{remark}
The summability properties assumed for $v$ and $u$ imply that the weak formulation \eqref{ncy0.6} is meaningful. In particular, $v\in L^2(\Omega)\cap L^\infty(\Omega)$ and $\nabla u\in L^2(\Omega)$ guarantee that the integral $\int_{\Omega} (v\cdot \nabla u) \cdot v\ dx$ is finite. Moreover, the integral $\int_{\Omega} [(\omega \times x)\cdot \nabla u- \omega\times u]\cdot v \ dx$ is finite because $(\omega \times x)\cdot \nabla u- \omega\times u \in L^2(\Omega)$. 
\end{remark}

Proposition \ref{P1} and the above remarks lead to the following definition of the Lagrangians:
\begin{equation}
\begin{array}{rcl}
\displaystyle \mathcal{L}_{\Gamma} (v,\vperp,\vstar,u,\zeta) & :=  & 
\displaystyle  \int_{\Omega}  |D(v)|^2 \ dx
+\frac{1}{4}\int_{\partial\Omega}(v_*+v_*^{\mathcal C})\cdot n |V+v_*+v_*^{\mathcal C}|^2\,d\gamma
\medskip \\
&& \displaystyle -2 \int_{\Omega}  D(v): D(u) \ dx
+ \int_{\Omega} (v\cdot \nabla u) \cdot v\ dx
\medskip \\
 && \displaystyle -\int_{\Omega} (V\cdot \nabla u- \omega\times u)\cdot v \ dx
+m (\xi\times \omega)\cdot \ell_u 
+((I \omega)\times \omega)\cdot k_u
\medskip \\
&& \displaystyle +\int_{\partial \Omega} \left((v_*+v_*^{\mathcal C})\cdot n \right) (\omega \times x)\cdot (\ell_u+k_u\times x) \ d \gamma
- \langle v-V- \vstar-v_*^{\mathcal C},\zeta\rangle_{\partial \Omega} 
\end{array}
\label{lagrangeNS}
\end{equation}
for $(v,\vperp,\vstar, u,\zeta)\in \mathcal{Y}\times  \mathcal{C}_\chi\times  \mathcal{V}_{\Gamma}\times \mathcal{U}\times \mathcal{Z},$ and \begin{equation}
\begin{array}{rcl}
\displaystyle {\mathcal L}_\tau (v,\vperp,\vstar,u,\zeta) & :=  & 
\displaystyle  \int_{\Omega}  |D(v)|^2 \ dx
 -2 \int_{\Omega}  D(v): D(u) \ dx
+ \int_{\Omega} (v\cdot \nabla u) \cdot v\ dx
\medskip \\
 && \displaystyle -\int_{\Omega} \left(V\cdot \nabla u- \omega\times u \right) \cdot v \ dx
+m (\xi\times \omega)\cdot \ell_u 
+((I \omega)\times \omega)\cdot k_u
\medskip \\
&& \displaystyle 
- \langle v-V- \vstar-v_*^{\mathcal C},\zeta\rangle_{\partial \Omega} 
\end{array}
\label{lagrangeNS2}
\end{equation}
for $(v,\vperp,\vstar, u,\zeta)\in \mathcal{Y}\times  \mathcal{C}_\tau\times  \mathcal{V}_\tau\times \mathcal{U}\times \mathcal{Z}.$

Let $\widehat v_*$ be a solution of the problem \eqref{Pchis}
and denote by 
$$(\widehat{v},\widehat v_*^{\mathcal C},\widehat p)\in \mathcal{Y}\times \mathcal{C}_\chi\times W^{1,2}(\Omega)$$  
the corresponding solution 
of \eqref{ncy0.0}--\eqref{ncy0.51} given by Theorem \ref{Tmain}.
We now obtain the adjoint system by considering the equations
\begin{equation}\label{lis0.7}
D_v \mathcal{L}_{\Gamma} (\widehat{v},\widehat v_*^{\mathcal C},\widehat{v}_*,\widehat{u},\widehat{\zeta}) v = 0 
\quad \forall v \in \mathcal{Y},
\end{equation}
\begin{equation}\label{lis0.7bis}
D_{v_*^{\mathcal C}} \mathcal{L}_{\Gamma} (\widehat{v},\widehat v_*^{\mathcal C},\widehat{v}_*,\widehat{u},\widehat{\zeta}) v_*^{\mathcal C} = 0 
\quad \forall v_*^{\mathcal C} \in \mathcal{C}_\chi
\end{equation}
for the unknowns $\widehat{u},\widehat{\zeta}$. In the case of tangential controls, equations \eqref{lis0.7} and \eqref{lis0.7bis} are replaced by
\begin{equation}
D_v {\mathcal L}_\tau (\widehat{v},\widehat v_*^{\mathcal C},\widehat{v}_*,\widehat{u},\widehat{\zeta}) v = 0 
\quad \forall v \in \mathcal{Y},
\label{derLtau}
\end{equation}
\begin{equation}\label{derLtaubis}
D_{v_*^{\mathcal C}} \mathcal{L}_{\tau} (\widehat{v},\widehat v_*^{\mathcal C},\widehat{v}_*,\widehat{u},\widehat{\zeta}) v_*^{\mathcal C} = 0 
\quad \forall v_*^{\mathcal C} \in \mathcal{C}_\tau.
\end{equation}

By computing the G\^ateaux derivatives of $\mathcal{L}_\Gamma$ or $\mathcal{L}_\tau$, we can rewrite the above equations. Such a calculation is standard but for sake of completeness, we give it in the case of \eqref{derLtau}:
we have to pass to the limit $h \to 0$ in
$$
\frac{{\mathcal L}_\tau (\widehat{v}+hv,\widehat v_*^{\mathcal C},\widehat{v}_*,\widehat{u},\widehat{\zeta}) - {\mathcal L}_\tau (\widehat{v},\widehat v_*^{\mathcal C},\widehat{v}_*,\widehat{u},\widehat{\zeta})}{h}
$$
where
$$
\begin{array}{rcl}
\displaystyle {\mathcal L}_\tau (\widehat{v}+hv,\widehat v_*^{\mathcal C},\widehat{v}_*,\widehat u,\widehat \zeta)  
& =  &
\displaystyle  \int_{\Omega}  |D(\widehat{v}+hv)|^2 \ dx
 -2 \int_{\Omega}  D(\widehat{v}+hv): D(\widehat u) \ dx
+ \int_{\Omega} ((\widehat{v}+hv)\cdot \nabla \widehat u) \cdot (\widehat{v}+hv)\ dx
\medskip \\
&& \displaystyle -\int_{\Omega} \left(V\cdot \nabla \widehat u- \omega\times \widehat u \right)\cdot (\widehat{v}+hv) \ dx
+m (\xi\times \omega)\cdot \ell_{\widehat u} 
+((I \omega)\times \omega)\cdot k_{\widehat u}
\medskip \\
\displaystyle 
&& \displaystyle - \langle \widehat{v}+hv -V- \widehat{v}_*-\widehat v_*^{\mathcal C},\widehat \zeta\rangle_{\partial \Omega} 
\end{array}
$$
and
$$
\begin{array}{rcl}
\displaystyle {\mathcal L}_\tau (\widehat{v},\widehat v_*^{\mathcal C},\widehat{v}_*,\widehat u,\widehat \zeta) & =  & 
\displaystyle  \int_{\Omega}  |D(\widehat{v})|^2 \ dx
 -2 \int_{\Omega}  D(\widehat{v}): D(\widehat u) \ dx
+ \int_{\Omega} (\widehat{v}\cdot \nabla \widehat u) \cdot \widehat{v}\ dx
\medskip \\
 && \displaystyle -\int_{\Omega} \left( V\cdot \nabla \widehat u- \omega\times \hat u \right) \cdot \widehat{v} \ dx
+m (\xi\times \omega)\cdot \ell_{\widehat u} 
+((I \omega)\times \omega)\cdot k_{\widehat u}
\medskip \\
&& \displaystyle 
- \langle \widehat{v}-V- \widehat{v}_*-\widehat v_*^{\mathcal C},\widehat \zeta\rangle_{\partial \Omega} .
\end{array}
$$
Simplifying the above expressions and letting $h\to 0$ in 
$$
\begin{array}{rcl}
&  & \displaystyle \frac{{\mathcal L}_\tau (\widehat{v}+hv,\widehat v_*^{\mathcal C},\widehat{v}_*,\widehat{u},\widehat{z}) - {\mathcal L}_\tau (\widehat{v},\widehat v_*^{\mathcal C},\widehat{v}_*,\widehat{u},\widehat{\zeta})}{h} \medskip \\
& = & \displaystyle 2  \int_{\Omega}  D(\hat v): D(v) \ dx + h \int_\Omega |D(v)|^2\medskip \\
&  & \displaystyle  + \int_{\Omega} \big(\widehat{v}\cdot \nabla \widehat{u}\big)\cdot v\ dx
+ \int_{\Omega} \big(v\cdot \nabla \widehat{u}\big)\cdot \widehat{v}\ dx + h \int_{\Omega} \big(v \cdot \nabla \widehat{u}\big)\cdot v\ dx \medskip \\
&  & \displaystyle -2 \int_{\Omega}  D(v): D(\widehat u) \ dx
 -\int_{\Omega} [(V\cdot \nabla) \widehat u- \omega\times \widehat u]\cdot v \ dx
- \langle v,\widehat \zeta \rangle_{\partial \Omega} 
\end{array}
$$
yields that \eqref{derLtau} is equivalent to
\begin{multline}\label{ae1NS}
 \displaystyle 
 2 \int_{\Omega} D(\widehat{v} - \widehat{u} ):D(v) dx 
+ \int_{\Omega} \big(\widehat{v}\cdot \nabla \widehat{u}\big)\cdot v\ dx 
+ \int_{\Omega} \big(v\cdot \nabla \widehat{u}\big)\cdot \widehat{v}\ dx \medskip \\ 
 -\int_{\Omega} \left( V\cdot \nabla \hat u- \omega\times \hat u \right) \cdot v \ dx
- \langle v,\hat \zeta \rangle_{\partial \Omega} =0, \quad \forall v \in \mathcal Y.
\end{multline}
By similar calculations, we see that the relation \eqref{lis0.7} is also equivalent to \eqref{ae1NS}, whereas \eqref{lis0.7bis} and \eqref{derLtaubis} can be respectively written as follows:
\begin{multline}
\frac{1}{4}\int_{\partial\Omega} v_*^{\mathcal C}\cdot n |V+\widehat v_*+\widehat  v_*^{\mathcal C}|^2\,d\gamma
+\frac{1}{2}\int_{\partial\Omega} (\widehat v_*+\widehat  v_*^{\mathcal C})\cdot n (V+\widehat v_*+\widehat  v_*^{\mathcal C})\cdot v_*^{\mathcal C}\,d\gamma
\\
+\int_{\partial \Omega} \left(v_*^{\mathcal C}\cdot n \right) (\omega \times x)\cdot \widehat u \ d \gamma
+ \langle \vperp,\widehat{\zeta} \rangle_{\partial \Omega} =0, \quad \forall \vperp\in  \mathcal{C}_\chi,\label{ae2NS} 
\end{multline}
and 
\begin{equation}
\langle \vperp,\widehat{\zeta} \rangle_{\partial \Omega} =0, \quad \forall \vperp\in  \mathcal{C}_\tau.\label{ae2NSter} 
\end{equation}
Once we have a solution $(\widehat u,\widehat\zeta)\in {\mathcal U}\times {\mathcal Z}$ to \eqref{ae1NS}, see \eqref{defcalu}--\eqref{test-bdy}, we deduce that there exists a pressure $\widehat{q}\in L^2_{\rm loc}(\overline\Omega)$ which together with $\widehat u$ obeys
\begin{gather}
-\div \sigma(\widehat{u} - \widehat{v},\widehat{q}-\widehat{p}) 
-   \widehat{v} \cdot \nabla \widehat{u}-    (\nabla \widehat{u} )^\top  \widehat{v} 
+ V\cdot \nabla \widehat{u}- \omega\times \widehat{u} =  0 \quad \text{in} \ \Omega \label{cs1NS}\\
\div \widehat{u} =0 \quad \text{in} \ \Omega \label{cs2NS} \\
\widehat{u} = \ell_{\widehat u}+k_{\widehat u}\times x \quad \text{on}\ \partial \Omega \label{cs3NS} \\
\lim_{|x|\to \infty} \widehat{u}(x) = 0. \label{cs4NS}
\end{gather}
Since $\widehat v\cdot\nabla\widehat u+(\nabla\widehat u)^\top\widehat v-\Delta\widehat v+\nabla\widehat p\in L^2(\Omega)\cap D^{-1,2}(\Omega)$, see Lemma \ref{L1} below, we employ Proposition \ref{P23} to see that the pressure $\widehat q\in W^{1,2}(\Omega)$ can be singled out.

Taking the scalar product of \eqref{cs1NS} with $\psi_Rv$, where $v\in \mathcal{Y}$ is arbitrary and $\psi_R$ is the same cut-off function as in the final stage of the proof of Theorem \ref{Tmain},
integrating by parts and letteing $R\to\infty$ (where $\widehat q\in L^2(\Omega)$ is used),
we deduce from \eqref{ae1NS} that
\begin{equation}\label{18:31}
\widehat{\zeta} =  \sigma (\widehat{v} - \widehat{u},\widehat{p}-\widehat{q})n.
\end{equation}
Recalling that  $\mathcal{Z} = L^2(\partial \Omega)$ and replacing \eqref{18:31} in \eqref{ae2NS}, yields
\begin{multline}\label{cs5NS}
\int_{\partial \Omega}
\big(\sigma( \widehat{v}-\widehat{u},\widehat{p}-\widehat{q})n\big) \cdot \vperp \ d\gamma
+\frac{1}{4}\int_{\partial\Omega} v_*^{\mathcal C}\cdot n |V+\widehat v_*+\widehat  v_*^{\mathcal C}|^2\,d\gamma
\\
+\frac{1}{2}\int_{\partial\Omega} (\widehat v_*+\widehat  v_*^{\mathcal C})\cdot n (V+\widehat v_*+\widehat  v_*^{\mathcal C})\cdot v_*^{\mathcal C}\,d\gamma
+\int_{\partial \Omega} \left(v_*^{\mathcal C}\cdot n \right) (\omega \times x)\cdot \widehat u \ d \gamma=0, \quad \forall \vperp\in  \mathcal{C}_\chi.
\end{multline}

In the case of tangential controls, \eqref{ae2NSter} takes the form
\begin{equation}\label{cs5NSbis}
\int_{\partial \Omega}\big(\sigma(\widehat{v}-\widehat{u},\widehat{p}-\widehat{q})n\big)\cdot\vperp \ d\gamma = 0, \quad \forall \vperp\in  \mathcal{C}_\tau.
\end{equation}

\subsection{Well-posedness of the adjoint system}
Now we show that the adjoint system \eqref{cs1NS}--\eqref{cs4NS} subject to \eqref{cs5NS}/\eqref{cs5NSbis} is well-posed.  With  $\widehat{v} \in {\mathcal Y}$ given, we can define the following mapping
\begin{equation}\label{deffrakF}
\mathfrak{F} = \mathfrak{F}_{\widehat{v}} : D^{1,2}(\Omega)   \to L^2(\Omega) \cap D^{-1,2}(\Omega),
\quad
u \mapsto \widehat{v} \cdot \nabla {u}  +  (\nabla {u} )^\top  \widehat{v}.
\end{equation}
\begin{lemma}\label{L1}
Assume that $\widehat{v} \in {\mathcal Y}$. Then
the mapping $\mathfrak{F}$ is well-defined and continuous:
$$
\| \mathfrak{F}(u)\|_{L^2(\Omega) \cap D^{-1,2}(\Omega)} \leq C \lceil \widehat{v} \rceil_{1,\varpi,\Omega} |u|_{1,2,\Omega}.
$$
\end{lemma}
\begin{proof}
Both terms in the formula of $\mathfrak{F}_{\widehat{v}}(u)$ can be handled in the same way.
First, it is immediate to obtain
$$
\|\widehat{v} \cdot \nabla {u}+   (\nabla {u} )^\top  \widehat{v}\|_{2,\Omega}\leq 
2\|\widehat{v}\|_{\infty,\Omega} \|\nabla u\|_{2,\Omega}
\leq 2 \lceil \widehat{v} \rceil_{1,\varpi,\Omega} | u |_{1,2,\Omega}.
$$
Now we notice that, for arbitrary $\varphi \in D^{1,2}_0(\Omega)$, by 
 the Hardy inequality,
$$
\langle \widehat{v} \cdot \nabla {u} +   (\nabla {u} )^\top  \widehat{v}, \varphi \rangle_{D^{-1,2}(\Omega),D^{1,2}_0(\Omega)} \leq   
C \lceil \widehat{v} \rceil_{1,\varpi,\Omega} \| \nabla u \|_{2,\Omega} \left\| \frac{\varphi} \varpi \right\|_{2,\Omega} \leq   
C \lceil \widehat{v} \rceil_{1,\varpi,\Omega} \| \nabla u \|_{2,\Omega} \| 
\nabla \varphi \|_{2,\Omega}, 
$$
where $\varpi(x)$ is given by \eqref{weight} and therefore
$$
|\widehat{v} \cdot \nabla {u}+   (\nabla {u} )^\top  \widehat{v}|_{-1,2,\Omega} 
\leq C \lceil \widehat{v} \rceil_{1,\varpi,\Omega}  | u |_{1,2,\Omega}.
$$
\end{proof}
\begin{theorem}\label{T1}
There exists a constant $\kappa_2\in (0,c_0]$ 
depending on $\Omega$
such that if $\xi,\omega\in \mathbb R^3$ and $v_*\in {\mathcal V}_\Gamma$ (resp. $v_*\in {\mathcal V}_\tau$) satisfy
\begin{equation}
|\xi|\leq\kappa_2, \qquad |\omega|\leq\kappa_2, \qquad
\|v_*\|_{3/2,2,\partial\Omega}\leq\kappa_2,
\label{small-2}
\end{equation}
then the adjoint system \eqref{cs1NS}--\eqref{cs4NS} subject to
\eqref{cs5NS} (resp. \eqref{cs5NSbis}) admits a unique solution 
$(\widehat{u},\widehat{q},\ell_{\widehat{u}},k_{\widehat{u}})\in \mathcal{U}\times W^{1,2}(\Omega)\times \mathbb R^3 \times\mathbb R^3$.
\end{theorem}
\begin{proof}
We show only the case of localized boundary values since the tangential case is similar.
Consider the space $$
\mathcal{B} := \left\{ u \in L^6(\Omega)\cap D^{1,2}(\Omega)\cap D^{2,2}(\Omega) \ ;\  
\nabla \cdot u = 0 \mbox{ in } \Omega 
\right\},
$$
 with the norm $\| u\|_{\mathcal{B}}:=|u|_{1,2,\Omega} + |u|_{2,2,\Omega}$.
 
 Assume $\overline{u} \in \mathcal{B}$ and consider the unique solution of the linear problem
\begin{gather}
-\div \sigma({u} ,{q}) + V\cdot \nabla {u}- \omega\times {u} =  
\mathfrak{F}(\overline{u})-\div \sigma(\widehat{v},\widehat{p})  \quad \text{in} \ \Omega  \label{1644}\\
\div {u} =0 \quad \text{in} \ \Omega  \\
{u} = {a} + {b} \times x \quad \text{on}\ \partial \Omega \\
\lim_{|x|\to \infty} {u}(x) = 0, \label{1645}
\\
\int_{\partial \Omega}\big(\sigma ({u},{q})n\big)\cdot \vperp \ d\gamma 
=
\int_{\partial \Omega} (\mathfrak{G}(\overline{u})+\mathfrak{H}) \cdot \vperp \ d\gamma , \quad \forall \vperp\in  \mathcal{C}_\chi,\label{1646}
\end{gather}
where 
$$
\mathfrak{G}(\overline{u}):=[(\omega \times x)\cdot \overline u] n,
$$
$$
\mathfrak{H}:=
\sigma (\widehat{v},\widehat{p})n +
\frac{1}{4} |V+\widehat v_*+\widehat  v_*^{\mathcal C}|^2 n + 
\frac{1}{2} (\widehat v_*+\widehat  v_*^{\mathcal C})\cdot n (V+\widehat v_*+\widehat  v_*^{\mathcal C}).
$$
and $\mathfrak{F}(\overline{u})$ is given by \eqref{deffrakF}. 

More precisely, the above linear system is solved by decomposing $u$ and $q$ as follows
\begin{equation}\label{19:12}
{u}=\sum_{i=1}^3 {a}_i {v}^{(i)}+\sum_{i=1}^3 {b}_i {V}^{(i)}+ {u}_f,\quad
{q}=\sum_{i=1}^3 {a}_i {q}^{(i)}+\sum_{i=1}^3 {b}_i {Q}^{(i)}+ {q}_f
\end{equation}
where $(v^{(i)},q^{(i)})$ and $(V^{(i)},Q^{(i)})$ are given by \eqref{313} and \eqref{323},
 whereas $(u_f,q_f)$ is the solution of
$$
\begin{array}{c}
-\div \sigma({u}_f ,{q}_f) + V\cdot \nabla {u}_f- \omega\times {u}_f =  
\mathfrak{F}(\overline{u})-\div \sigma(\widehat{v},\widehat{p})  \quad \text{in} \ \Omega \medskip \\
\div {u}_f =0 \quad \text{in} \ \Omega  \medskip \\
{u}_f = 0 \quad \text{on}\ \partial \Omega \medskip \\
{\displaystyle \lim_{|x|\to \infty} {u}_f(x) = 0},
\end{array}
$$
obtained in Proposition \ref{P23}.
Then $(u,p)$ automatically satisfies \eqref{1644}--\eqref{1645}. It only remains to choose $(a,b)$ such that \eqref{1646} holds. 
Using the basis \eqref{controlspace} of ${\mathcal C}_\chi$, we reduce \eqref{1646} to
\begin{equation}\label{1646-bis}
\int_{\partial \Omega} \big(\sigma({u},{q})n\big)\cdot \chi g^{(i)} \ d\gamma 
=
\int_{\partial \Omega} (\mathfrak{G}(\overline{u})+\mathfrak{H})\cdot \chi g^{(i)} \ d\gamma 
\quad (i=1,2,3),
\end{equation}
\begin{equation}\label{1646-ter}
\int_{\partial \Omega}\big(\sigma ({u},{q})n\big) \cdot \chi G^{(i)} \ d\gamma 
=
\int_{\partial \Omega} (\mathfrak{G}(\overline{u})+\mathfrak{H})\cdot \chi G^{(i)} \ d\gamma 
\quad (i=1,2,3),
\end{equation}
which, in view of \eqref{19:12} together with \eqref{014}--\eqref{024}, can be written as
\begin{equation}\label{linlin0}
A \begin{bmatrix}
a\\ b
\end{bmatrix}
= \begin{bmatrix}
c\\ d
\end{bmatrix}
\end{equation}
where
\begin{gather}
\label{cs5NS-ter}
A_{i,j}:=\int_{\partial \Omega}  \chi {g}^{(i)} \cdot {g}^{(j)} \ d\gamma \quad (i,j\leq 3),
\quad
A_{i,j}:=\int_{\partial \Omega}  \chi {g}^{(i)} \cdot {G}^{(j-3)} \ d\gamma \quad (i\leq 3, j\geq 4),
\\
\label{cs5NS-ter-ter}
A_{i,j}:=\int_{\partial \Omega}  \chi {G}^{(i-3)} \cdot {g}^{(j)} \ d\gamma \quad (i\geq 4,j\leq 3),
\quad
A_{i,j}:=\int_{\partial \Omega}  \chi {G}^{(i-3)} \cdot {G}^{(j-3)} \ d\gamma \quad (i, j\geq 4),
\end{gather}
\begin{gather}
\label{ncy3.4}
c_{i}:=\int_{\partial \Omega} \left(\mathfrak{G}(\overline{u})+\mathfrak{H}- \sigma ({u}_f,{q}_f)n\right) \cdot \chi {g}^{(i)} \ d\gamma 
\quad (i=1,2,3),
\\
\label{ncy3.4-ter}
d_{i}:=\int_{\partial \Omega} \left(\mathfrak{G}(\overline{u})+\mathfrak{H}- \sigma ({u}_f,{q}_f)n\right) \cdot \chi {G}^{(i)} \ d\gamma 
\quad (i=1,2,3).
\end{gather}
 By \cite[Lemma 4.3]{HST} we know that, for $\xi$ and $\omega$ satisfying \eqref{icd0.0}, the matrix $A$ is invertible; in fact, this is needed and even crucial in Theorem \ref{Tmain} although it is hidden in \cite[Proposition 4.5]{HST}. Thus, we obtain the existence and uniqueness of $(a,b)$ 
satisfying \eqref{linlin0}--\eqref{ncy3.4-ter} and deduce that for any 
$\overline{u}\in  {\mathcal B}$  
there exists a unique solution
$(u,q,a,b)\in {\mathcal B}\times W^{1,2}(\Omega)\times \mathbb R^3\times \mathbb R^3$
of \eqref{1644}--\eqref{1646}.
By the equation \eqref{1644} we have $V\cdot\nabla u-\omega\times u\in L^2(\Omega)$ as well, so that $u\in {\mathcal U}$.

To obtain the existence of a solution of the adjoint system \eqref{cs1NS}--\eqref{cs4NS} and \eqref{cs5NS}, we only need to show that the mapping
$$
\Xi : {\mathcal B} \to {\mathcal B}, \quad \overline{u} \mapsto u
$$
is contractive as long as $(\xi,\omega)$ and $v_*$ are
small enough. 
Using the linearity of the adjoint system, it is sufficient to consider the system
$$
\begin{array}{c}
-\div \sigma({u} ,{q}) + V\cdot \nabla {u}- \omega\times {u} =  
\mathfrak{F}(\overline{u})  \quad \text{in} \ \Omega  \medskip \\
\div {u} =0 \quad \text{in} \ \Omega \medskip \\
{u} = {a} + {b} \times x \quad \text{on}\ \partial \Omega \medskip \\
{\displaystyle \lim_{|x|\to \infty} {u}(x) = 0}, 
\end{array}
$$
with 
\begin{equation}\label{linlin}
 \begin{bmatrix}
a\\ b
\end{bmatrix}
= A^{-1} \begin{bmatrix}
c\\ d
\end{bmatrix}
\end{equation}
where
\begin{gather}
\label{ncy3.4-bis}
c_{i}:=\int_{\partial \Omega} \left(\mathfrak{G}(\overline{u})- \sigma ({u}_f,{q}_f)n\right) \cdot \chi {g}^{(i)} \ d\gamma 
\quad (i=1,2,3),
\\
\label{ncy3.4-bis-ter}
d_{i}:=\int_{\partial \Omega} \left(\mathfrak{G}(\overline{u})- \sigma ({u}_f,{q}_f)n\right) \cdot \chi {G}^{(i)} \ d\gamma 
\quad (i=1,2,3)
\end{gather}
and 
$$
\begin{array}{c}
-\div \sigma({u}_f ,{q}_f) + (V\cdot \nabla) {u}_f- \omega\times {u}_f =  
\mathfrak{F}(\overline{u})  \quad \text{in} \ \Omega \medskip \\
\div {u}_f =0 \quad \text{in} \ \Omega \medskip \\
{u}_f = 0 \quad \text{on}\ \partial \Omega \medskip \\
{\displaystyle \lim_{|x|\to \infty} {u}_f(x) = 0}.
\end{array}
$$
Using the trace theorem and Proposition \ref{P23}, we have 
$$
|(a,b)| \leq C \left(\|\mathfrak{F}(\overline{u})\|_{L^2(\Omega) \cap D^{-1,2}(\Omega)}+|\omega| |\overline{u}|_{1,2,\Omega}\right)
$$
and thus, using again Proposition \ref{P23} and Lemma \ref{L1} we deduce
$$
\| \Xi (\overline{u})\|_{\mathcal B} = \|u\|_{\mathcal B} \leq C \left( \|\mathfrak{F}(\overline{u})\|_{L^2(\Omega) \cap D^{-1,2}(\Omega)}+|(a,b)| \right)
 \leq 
C \left(\lceil \widehat{v} \rceil_{1,\varpi,\Omega} + |\omega| \right) |\overline{u}|_{1,2,\Omega}
\leq C(R_0+|\omega|)\|\overline{u}\|_{\mathcal B},
$$
where $R_0$ is given by \eqref{inv-ball}. This yields the existence and uniqueness of a solution of the adjoint system \eqref{cs1NS}--\eqref{cs4NS}  subject to \eqref{cs5NS} under the condition \eqref{small-2} with $\kappa_2\in (0,c_0]$ small enough.
\end{proof}

\subsection{Optimality condition}
Recall the mappings $\Lambda_\tau : \vstar \mapsto (v,v_*^{\mathcal C})$ and $\Lambda_\chi : \vstar \mapsto (v,v_*^{\mathcal C})$ defined by \eqref{lis1.2} and \eqref{lis1.2chi}, respectively. For both cases, we abbreviate them to $\Lambda$ and, similarly, we write $\mathcal{L}$ instead of $\mathcal{L}_{\tau}$ or $\mathcal{L}_{\Gamma}$.
Because of Proposition \ref{P1}, the functional \eqref{icd0.3} can be written as
\begin{equation}\label{lis0.4}
J(\vstar)=
 2\,
\mathcal{L}(\Lambda(\vstar), \vstar,u,\zeta),
\end{equation}
 no matter which $(u,\zeta)\in {\mathcal U}\times {\mathcal Z}$ may be. Assume \eqref{small-2} and let us take, in particular, the solution $\widehat u\in {\mathcal U}$ to the adjoint system \eqref{cs1NS}--\eqref{cs4NS} subject to \eqref{cs5NS}/\eqref{cs5NSbis} obtained in Theorem \ref{T1} together with $\widehat\zeta\in {\mathcal Z}$ given by \eqref{18:31} so that \eqref{lis0.7}--\eqref{lis0.7bis} or \eqref{derLtau}--\eqref{derLtaubis} are satisfied.

We are now in a position to provide the optimality conditions for problems
\eqref{Ptaus} and \eqref{Pchis}:
\begin{theorem}\label{main}
Let $\Omega$ be of class $C^3$. Set $\kappa_0:=\min\{\kappa_1, \kappa_2\}$, where $\kappa_1$ and $\kappa_2$ are the constants in Theorem \ref{zhz} and Theorem \ref{T1}, respectively. Suppose that $\xi, \omega\in\mathbb R^3$ and $v_*\in {\mathcal V}_\Gamma$ (resp. ${\mathcal V}_\tau$) satisfy
\[
|\xi|\leq\kappa_0, \qquad |\omega|\leq\kappa_0, \qquad
\|v_*\|_{3/2,2,\partial\Omega}\leq\kappa_0.
\]
Let $\widehat v_*$ be a solution of the optimal control problem \eqref{Pchis} (resp. \eqref{Ptaus})
with $\kappa=\kappa_0$,
$(\widehat v,\widehat v_*^{\mathcal C},\widehat p)$ the corresponding state obtained in Theorem \ref{Tmain}
and $(\widehat u, \widehat q)$ the solution of the adjoint system
\eqref{cs1NS}--\eqref{cs4NS} subject to \eqref{cs5NS} (resp. \eqref{cs5NSbis}) obtained in Theorem \ref{T1}.
Then we have
\begin{multline}\label{lis0.9}
\int_{\partial \Omega}
\big(\sigma (\widehat{v} - \widehat{u},\widehat{p}-\widehat{q})n\big) \cdot (\vstar-\widehat v_* ) \ d\gamma 
+\frac{1}{4}\int_{\partial\Omega} (\vstar-\widehat v_* )\cdot n |V+\widehat v_*+\widehat  v_*^{\mathcal C}|^2\,d\gamma
\\
+\frac{1}{2}\int_{\partial\Omega} (\widehat v_*+\widehat  v_*^{\mathcal C})\cdot n (V+\widehat v_*+\widehat  v_*^{\mathcal C})\cdot (\vstar-\widehat v_* )\,d\gamma
+\int_{\partial \Omega} \left((\vstar-\widehat v_* )\cdot n \right) (\omega \times x)\cdot \widehat u \ d \gamma\geq 0, \quad \forall v_*\in  \mathcal{V}_{\Gamma}^{\kappa_0},
\end{multline}
in the case of localized controls, while
\begin{equation}\label{lis0.9der}
\int_{\partial \Omega} 
\big(\sigma(\widehat{v} - \widehat{u},\widehat{p}-\widehat{q})n\big) \cdot (\vstar-\widehat v_* ) \ d\gamma 
\geq 0, \quad \forall v_*\in  \mathcal{V}_{\tau}^{\kappa_0},
\end{equation}
in the case of tangential controls.
\end{theorem}
\begin{proof}
We begin by showing that
\begin{equation}\label{lis0.5}
 \frac{1}{2}
D_{\vstar} J(\widehat v_*) \vstar
=D_{(v,v_*^{\mathcal C})} \mathcal{L}(\Lambda(\widehat v_* ), \widehat v_* , \widehat{u}, \widehat{\zeta}) 
	D_{\vstar} \Lambda(\widehat v_* ) \vstar
+ D_{\vstar} \mathcal{L}(\Lambda(\widehat v_* ), \widehat v_* , \widehat{u}, \widehat{\zeta}) \vstar,
\end{equation}
where the direction $v_*\in {\mathcal V}_\Gamma$ must be taken such that $\langle\widehat v_*, v_*\rangle_{W^{3/2,2}(\partial\Omega)}<0$ if $\|\widehat v_*\|_{3/2,2,\partial\Omega}=\kappa_0$, whereas it can be arbitrary if $\|\widehat v_*\|_{3/2,2,\partial\Omega}<\kappa_0$.

In order to compute
$$
D_{\vstar} J(\widehat v_*) \vstar  = \lim_{h \to 0} \frac{J(\widehat v_*+ h \vstar )  - J(\widehat v_*) }{h} 
$$
we write
$$
\begin{array}{rcl}
\displaystyle
\frac{1}{2} \,\frac{J(\widehat v_*+ h \vstar )  - J(\widehat v_*) }{h} 
& =  & \displaystyle \frac{ \mathcal{L}(\Lambda(\widehat v_*+ h \vstar ), \widehat v_*+ h \vstar , \widehat{u}, \widehat{\zeta})  - \mathcal{L}(\Lambda(\widehat v_*), \widehat v_*, \widehat{u}, \widehat{\zeta})}{h}  
\medskip \\
& = & \displaystyle \frac{ \mathcal{L}(\Lambda(\widehat v_*+ h \vstar ), \widehat v_*+ h \vstar , \widehat{u}, \widehat{\zeta}) -  \mathcal{L}(\Lambda(\widehat v_*), \widehat v_*+ h \vstar , \widehat{u}, \widehat{\zeta})  }{h} 
\medskip \\
&& \displaystyle +
\frac{\mathcal{L}(\Lambda(\widehat v_*), \widehat v_*+ h \vstar , \widehat{u}, \widehat{\zeta}) - \mathcal{L}(\Lambda(\widehat v_*), \widehat v_*, \widehat{u}, \widehat{\zeta})}{h} 
\end{array}
$$
and, by the definition of Gateaux derivative, it is clear that
$$
\lim_{h\to 0}\frac{\mathcal{L}(\Lambda(\widehat v_*), \widehat v_*+ h \vstar , \widehat{u}, \widehat{\zeta}) - \mathcal{L}(\Lambda(\widehat v_*), \widehat v_*, \widehat{u}, \widehat{\zeta})}{h} 
= D_{\vstar} \mathcal{L}(\Lambda(\widehat v_* ), \widehat v_* , \widehat{u}, \widehat{\zeta}) \vstar.
$$
The relation  
$$
\Lambda(\widehat v_*+ h \vstar ) = \Lambda(\widehat v_*) + h D_{\vstar} \Lambda(\widehat v_* ) \vstar + o(h)= \Lambda(\widehat v_*) + h (\widehat{z}_h,\widehat{z}_{h*}^{\mathcal C}),
$$
with $(\widehat{z}_h,\widehat{z}_{h*}^{\mathcal C}) := D_{\vstar} \Lambda(\widehat v_* ) \vstar + o(h)/h $, yields
$$
\begin{array}{rcl}
&&\displaystyle\frac{ \mathcal{L}(\Lambda(\widehat v_*+ h \vstar ), \widehat v_*+ h \vstar , \widehat{u}, \widehat{\zeta}) -  \mathcal{L}(\Lambda(\widehat v_*), \widehat v_*+ h \vstar , \widehat{u}, \widehat{\zeta})  }{h} \medskip \\
&= &\displaystyle\frac{ \mathcal{L}(\Lambda(\widehat v_*) + h (\widehat{z}_h,\widehat{z}_{h*}^{\mathcal C}) , \widehat v_*+ h \vstar , \widehat{u}, \widehat{\zeta}) -  \mathcal{L}(\Lambda(\widehat v_*), \widehat v_*+ h \vstar , \widehat{u}, \widehat{\zeta})  }{h} \medskip \\
&= &\displaystyle D_{(v,v_*^{\mathcal C})} \mathcal{L}(\Lambda(\widehat v_* ) , \widehat v_*+ h \vstar , \widehat{u}, \widehat{\zeta}) (\widehat{z}_h,\widehat{z}_{h*}^{\mathcal C}) + \frac{o(h)}{h} .
\end{array}
$$
Therefore
$$
\lim_{h \to 0} \frac{ \mathcal{L}(\Lambda(\widehat v_*+ h \vstar ), \widehat v_*+ h \vstar , \widehat{u}, \widehat{\zeta}) -  \mathcal{L}(\Lambda(\widehat v_*), \widehat v_*+ h \vstar , \widehat{u}, \widehat{\zeta})  }{h} =
\lim_{h \to 0} D_{(v,v_*^{\mathcal C})} \mathcal{L}(\Lambda(\widehat v_* ) , \widehat v_*+ h \vstar , \widehat{u}, \widehat{\zeta})(\widehat{z}_h,\widehat{z}_{h*}^{\mathcal C})  .
$$
Note that, by Theorem \ref{zhz}, the convergence 
$$(\widehat{z}_h,\widehat{z}_{h*}^{\mathcal C}) =  \frac{\Lambda(\widehat v_*+ h \vstar ) -  \Lambda(\widehat v_*)}{h} \to D_{\vstar} \Lambda(\widehat v_* ) \vstar =: (\widehat{z},\widehat{z}_{*}^{\mathcal C}) \quad (h \to 0)$$  
means that
\begin{equation}\| \widehat{z}_h - \widehat{z} \|_{2,2,\Omega} + \lceil  \widehat{z}_h - \widehat{z}  \rceil_{1,\varpi,\Omega} + \| \widehat{z}_{h*}^{\mathcal C} - \widehat{z}_{*}^{\mathcal C}  \|_{3/2,2,\partial \Omega} \to 0.\label{strongJ2}\end{equation}

In the case of localized controls, we have
\begin{multline}
D_{v} \mathcal{L}_{\Gamma}(\Lambda(\widehat v_* ) , \widehat v_*+ h \vstar , \widehat{u}, \widehat{\zeta}) \widehat{z}_h = 2 \int_{\Omega} D(\widehat{v} - \widehat{u} ):D(\widehat{z}_h) dx 
+ \int_{\Omega} \big(\widehat{v}\cdot \nabla \widehat{u} \big)\cdot \widehat{z}_h\ dx
+ \int_{\Omega} \big(\widehat{z}_h\cdot \nabla \widehat{u}\big)\cdot \widehat{v}\ dx
\\
-\int_{\Omega} [(V\cdot \nabla) \widehat{u}- \omega\times \widehat{u}]\cdot \widehat{z}_h \ dx
- \int_{\partial \Omega} \widehat{z}_h\cdot \widehat{\zeta} \ ds \quad (\widehat{z}_h \in  \mathcal{Y}),\label{ae02NS}
\end{multline}
\begin{multline}
D_{v_*^{\mathcal C}} \mathcal{L}_{\Gamma}(\Lambda(\widehat v_* ) , \widehat v_*+ h \vstar  , \widehat{u}, \widehat{\zeta}) \widehat{z}_{h*}^{\mathcal C} = \frac{1}{4}\int_{\partial\Omega} \widehat{z}_{h*}^{\mathcal C} \cdot n |V + \widehat v_*+ h \vstar  + \widehat  v_*^{\mathcal C}|^2\,d\gamma \\
+\frac{1}{2}\int_{\partial\Omega} (\widehat v_*+ h \vstar  + \widehat  v_*^{\mathcal C})\cdot n (V + \widehat v_*+ h \vstar  +\widehat  v_*^{\mathcal C})\cdot \widehat{z}_{h*}^{\mathcal C} \,d\gamma \\
+\int_{\partial \Omega} \left(\widehat{z}_{h*}^{\mathcal C}\cdot n \right) (\omega \times x)\cdot \widehat u \ d \gamma
+ \langle \widehat{z}_{h*}^{\mathcal C},\widehat{\zeta} \rangle_{\partial \Omega} \quad (\widehat{z}_{h*}^{\mathcal C}   \in  \mathcal{C}_\chi), \label{ae2NSbis} 
\end{multline}
and analogously for $\mathcal{L}_\tau$, so that by taking the limit $h \to 0$ in \eqref{ae02NS} and \eqref{ae2NSbis}, and using \eqref{strongJ2}, we get
$$
\lim_{h \to 0} \frac{ \mathcal{L}(\Lambda(\widehat v_*+ h \vstar ), \widehat v_*+ h \vstar , \widehat{u}, \widehat{\zeta}) -  \mathcal{L}(\Lambda(\widehat v_*), \widehat v_*+ h \vstar , \widehat{u}, \widehat{\zeta})  }{h} = 
D_{(v,v_*^{\mathcal C})} \mathcal{L}(\Lambda(\widehat v_* ) , \widehat v_*, \widehat{u}, \widehat{z}) D_{\vstar} \Lambda(\widehat v_* ) \vstar.$$

Having shown \eqref{lis0.5} and using the definition of the adjoint system, specifically \eqref{lis0.7}--\eqref{lis0.7bis} or \eqref{derLtau}--\eqref{derLtaubis}
we get 
\begin{equation}\label{lis0.8}
 \frac 12 D_{\vstar} J(\widehat v_* ) \vstar
=D_{\vstar} \mathcal{L}(\Lambda(\widehat v_* ), \widehat v_* , \widehat{u}, \widehat{\zeta}) \vstar.
\end{equation}

Now, using that $\mathcal{V}_\tau^{\kappa_0}$ and $\mathcal{V}_{\Gamma}^{\kappa_0}$ are convex sets, we deduce from $J(\widehat v_*+h(v_*-\widehat v_*)\geq J(\widehat v_*)$ for every $v_*\in \mathcal{V}_\tau^{\kappa_0}$ or $\mathcal{V}_\Gamma^{\kappa_0}$ and $h\in (0,1)$ that
\begin{equation}\label{tt02}
\frac 12  D_{\vstar} J(\widehat v_* ) (\vstar-\widehat v_* )
=D_{\vstar} \mathcal{L}_\tau(\Lambda_\tau(\widehat v_* ), \widehat v_* , \widehat{u}, \widehat{\zeta}) (\vstar-\widehat v_* )\geq 0, \qquad \forall \vstar\in \mathcal{V}_\tau^{\kappa_0},
\end{equation}
and that
\begin{equation}\label{tt03}
 \frac 12  D_{\vstar} J(\widehat v_* ) (\vstar-\widehat v_* )
=D_{\vstar} \mathcal{L}_\Gamma(\Lambda_\chi(\widehat v_* ), \widehat v_* , \widehat{u}, \widehat{\zeta}) (\vstar-\widehat v_* )\geq 0,
\qquad \forall \vstar\in \mathcal{V}_\Gamma^{\kappa_0}.
\end{equation}

Using \eqref{lagrangeNS}, \eqref{lagrangeNS2} and \eqref{18:31}, we deduce \eqref{lis0.9}
and \eqref{lis0.9der}.
\end{proof}

Note that in Theorem \ref{main}, if in particular $\|\widehat v_*\|_{3/2,2,\partial\Omega}< \kappa_0$, then conditions \eqref{lis0.9} and \eqref{lis0.9der} become respectively
\begin{equation}\label{lis1.01}
\left[\sigma(\widehat{v}-\widehat{u},\widehat{p}-\widehat{q})n
+\frac{1}{4} |V+\widehat v_*+\widehat  v_*^{\mathcal C}|^2 n
\\
+\frac{1}{2}\big((\widehat v_*+\widehat  v_*^{\mathcal C})\cdot n\big) (V+\widehat v_*+\widehat  v_*^{\mathcal C})
+((\omega \times x)\cdot \widehat u ) n\right]
 \perp \mathcal{V}_{\Gamma}
\end{equation}
and
\begin{equation}\label{lis1.0}
\sigma(\widehat{v}-\widehat{u},\widehat{p}-\widehat{q})n \perp \mathcal{V}_\tau
\end{equation}
in $L^2(\partial\Omega)$.


\begin{thebibliography}{99}

\bibitem{Bog}
M.E. Bogovskii,
Solution of the first boundary value problem for the equation
of continuity of an incompressible medium,
{\it Sov. Math. Dokl.}
{\bf 20} (1979), 1094--1098.

\bibitem{Crispo} F. Crispo and P. Maremonti,
An interpolation inequality in exterior domains. 
{\it Rend. Sem. Mat. Univ. Padova}, {\bf 112} (2004), 11-39 

\bibitem{D} L. Ded\`e,
Optimal flow control for Navier-Stokes equations: drag minimization,
{\it Int. J. Numer. Meth. Fl.}, {\bf 55}:4 (2007), 347--366. 

\bibitem{FaHi}
R. Farwig and T. Hishida,
Asymptotic profle of steady Stokes flow around a rotating obstacle,
{\it Manuscripta Math.}
{\bf 136} (2011), 315--338.

\bibitem{FGH0}  A. V. Fursikov, M. D. Gunzburger and L. S. Hou, Boundary
value problems and optimal boundary control for the Navier-Stokes system:
the two-dimensional case, \emph{SIAM J. Control Optim.}, \textbf{36} (1998), No. 3,
852-894.

\bibitem{FGH}
A. V. Fursikov, M. D. Gunzburger and L. S. Hou, Optimal boundary control for the evolutionary Navier-Stokes system: the three-dimensional case, {\it SIAM J. Control Optim.} {\bf  43} (2005), No. 6, 2191-2232.

\bibitem{Galdi1997}
G.P. Galdi, 
On the Steady, Translational Self-Propelled Motion of a Symmetric
Body in a Navier-Stokes Fluid, Quaderni di Matematica della II Universita di Napoli, Vol. 1 (1997), 97--169.

\bibitem{Galdi1999}
G.P. Galdi, 
On the steady self-propelled motion of a body in a viscous incompressible fluid. {\it Arch. Rational Mech. Anal.} 148 (1999), 53--88.

\bibitem{GRev} G. P. Galdi, On the Motion of a Rigid Body in a Viscous Liquid: A Mathematical Analysis with Applications, {\it Handbook of Mathematical Fluid Dynamics}, Vol. 1, (2002) 655--679.

 
\bibitem{G}
G.P. Galdi,
{\it An Introduction to the Mathematical Theory of the Navier-Stokes Equations},
{\it Steady-State Problem},
Second Edition, Springer, 2011.

\bibitem{GS1}
G.P. Galdi and A.L. Silvestre,
The steady motion of a Navier-Stokes liquid around a rigid body,
{\it Arch. Rational Mech. Anal.} {\bf 184} (2007), 371--400.


\bibitem{HST}
T. Hishida, A.L. Silvestre and T. Takahashi,
A boundary control problem for the steady self-propelled motion
of a rigid body in a Navier-Stokes fluid,
{\it Ann. I.H.Poincar\'e, Analyse Non Lin\'eaire}
{\bf 34} (2017), 1507--1541.

\bibitem{KS1991}
H. Kozono and  H. Sohr, 
New a priori estimates for the Stokes equations in exterior domains, 
{\it Indiana U. Math. J.}, {\bf 40} (1991), No.1, 1--27. 

\bibitem{Ky}
M. Kyed,
Asymptotoic profile of a lineraized Navier-Stokes flow past a rotating body,
{\it Q. Appl. Math.}
{\bf 71} (2013), 489--500.

\bibitem{MTT}
J. San Mart\'in, T. Takahashi, M. Tucsnak, A control theoretic approach to the swimming of microscopic organisms, {\it Q. Appl. Math.}, {\bf 65}
(2007), 405--424.

\end{thebibliography}
\end{document}